\documentclass[12pt,a4paper]{amsart}

\usepackage[abbrev]{amsrefs}

\theoremstyle{plain}
  \newtheorem{thm}{Theorem}[section]
  \newtheorem{lem}[thm]{Lemma}
  
  \newtheorem{cor}[thm]{Corollary}
  \newtheorem{prop}[thm]{Proposition}

\theoremstyle{definition}
  \newtheorem{defn}[thm]{Definition}

  \newtheorem{asmp}[thm]{Assumption}

\theoremstyle{remark}
  \newtheorem{rem}[thm]{Remark}

  \newtheorem*{ack}{Acknowledgment}

\numberwithin{equation}{section}

\DeclareMathOperator{\vol}{vol}
\DeclareMathOperator{\diam}{diam}
\DeclareMathOperator{\supp}{supp}

\DeclareMathOperator{\ObsVar}{ObsVar}
\DeclareMathOperator{\Var}{Var}
\DeclareMathOperator{\Sep}{Sep}

\DeclareMathOperator{\Ric}{Ric}
\DeclareMathOperator{\Hess}{Hess}

\DeclareMathOperator{\Image}{Im}
\DeclareMathOperator{\IC}{IC}
\DeclareMathOperator{\ICL}{ICL}

\DeclareMathOperator{\RCD}{RCD}

\newcommand{\Leb}[1]{\mathcal L^1|_{#1}}

\newcommand{\cL}{\mathcal{L}}

\newcommand{\cV}{\mathcal{V}}

\newcommand{\field}[1]{\mathbb{#1}}

\newcommand{\R}{\field{R}}

\begin{document}

\title[Isoperimetric rigidity]
{Isoperimetric rigidity\\ and distributions of $1$-Lipschitz functions}

\thanks{The second named author is partially supported by a Grant-in-Aid
for Scientific Research from the Japan Society for the Promotion of Science}

\begin{abstract}
We prove that if a geodesic metric measure space satisfies
a comparison condition for isoperimetric profile
and if the observable variance is maximal, then the space is foliated by
minimal geodesics, where the observable variance is defined
to be the supremum of the variance of $1$-Lipschitz functions on the space.
Our result can be considered as a variant of Cheeger-Gromoll's splitting theorem
and also of Cheng's maximal diameter theorem.
As an application, we obtain a new isometric splitting theorem for a
complete weighted Riemannian manifold with a positive
Bakry-\'Emery Ricci curvature.
\end{abstract}

\author{Hiroki Nakajima}
\author{Takashi Shioya}

\address{Mathematical Institute, Tohoku University, Sendai 980-8578,
  JAPAN}
\email{hiroki.nakajima.s4@dc.tohoku.ac.jp}
\email{shioya@math.tohoku.ac.jp}

\date{\today}

\keywords{isoperimetric profile, metric measure space, concentration of measure, observable variance, Lipschitz function}

\maketitle

\section{Introduction}
\label{sec:intro}

A rigidity theorem in Riemannian geometry claims that
if a space is as large (in suitable sense) as a model space defined by a lower bound of
curvature of the space, then the structure of the space is determined.
For instance, Cheng's maximal diameter theorem \cite{Ch:eigen-comp} and
Cheeger-Gromoll's splitting theorem \cite{CG:splitting}
are two of the most celebrated rigidity theorems.
Recently, there are several works done for comparison of isoperimetric profile
under a lower Ricci curvature bound, i.e.,
if the Ricci curvature is bounded below for a complete Riemannian manifold,
or more generally if the Riemannian curvature-dimension condition due to
Ambrosio-Gigli-Savar\'e \cite{AGS:RiemRic} for a metric measure space is satisfied,
then the isoperimetric profile of the space is greater than or equal to
that of a model space
\cite{Gmv:green, BkLd:LG-isop, EMlm:sharp, AbMd:Gaussian-isop, CvMd:sharp-isop}.
In this paper, we prove a rigidity theorem for a metric measure space
under a comparison condition of isoperimetric profile
instead of the lower boundedness of Ricci curvature.
Since the comparison condition of isoperimetric profile
is much weaker than the lower boundedness of Ricci curvature,
we are not able to expect the same result as the maximal diameter theorem
nor the splitting theorem.
We introduce the \emph{observable variance} of the space,
which is a quantity to measure the largeness of a metric measure space.
We prove that, under the comparison condition of isoperimetric profile,
the observable variance has a certain upper bound,
and that, if it is maximal, then
we obtain a foliation structure by minimal geodesics of the space.
As an application, we obtain an isometric splitting theorem
for a complete weighted Riemannian manifold with
a positive Bakry-\'Emery Ricci curvature.

Throughout this paper, a \emph{metric measure space} $X$, or an \emph{mm-space} for short,
is a space equipped
with a complete separable metric $d_X$ and a Borel probability measure $\mu_X$.
Let $X$ be an mm-space.  The \emph{boundary measure} of a Borel set $A \subset X$ is
defined to be
\[
\mu_X^+(A) := \limsup_{\varepsilon \to 0+} \frac{\mu_X(U_\varepsilon(A)) - \mu_X(A)}
{\varepsilon},
\]
where $U_\varepsilon(A)$ denotes the open $\varepsilon$-neighborhood of $A$.
Denote by $\Image \mu_X$ the set of $\mu_X(A)$ for all Borel sets $A \subset X$.
The \emph{isoperimetric profile} $I_X : \Image \mu_X \to [\,0,+\infty\,)$ of $X$
is defined by
\[
I_X(v) := \inf\{\;\mu_X^+(A) \mid A \subset X \ \text{: Borel}, \ \mu_X(A) = v\;\}
\]
for $v \in \Image\mu_X$.
\begin{defn}[Isoperimetric comparison condition]
  We say that $X$ satisfies the \emph{isoperimetric comparison condition} $\IC(\nu)$
  for a Borel probability measure $\nu$ on $\R$ if
   \[
   I_X \circ V \ge V' \quad\text{$\mathcal L^1$-a.e. on $V^{-1}(\Image\mu_X)$},
   \]
   where $V$ denotes the cumulative distribution function of $\nu$
   and $\cL^1$ the one-dimensional Lebesgue measure on $\R$.
\end{defn}

In the case where $\nu$ and $\cL^1$
are absolutely continuous with each other,
$\IC(\nu)$ is equivalent to 
\begin{equation} \label{eq:IC-abc}
I_X \ge V' \circ V^{-1} \quad\text{$\mathcal L^1$-a.e. on $\Image\mu_X$},
\end{equation}
where $V' \circ V^{-1}$ coincides with the isoperimetric profile of $(\R,\nu)$
restricted to sets $A = (\,-\infty,a\,]$.
\eqref{eq:IC-abc} was formerly considered in \cite{Ld:conc,EMlm:sharp}.

Let $\lambda : [\,0,+\infty\,) \to [\,0,+\infty\,)$ be a strictly monotone increasing 
continuous function.
We define the \emph{$\lambda$-observable variance} $\ObsVar_\lambda(X)$ of $X$
to be the supremum of the \emph{$\lambda$-variance} of $f$,
\[
\Var_\lambda(f) := \int_X\int_X \lambda(|f(x)-f(x')|)\; d\mu_X(x)d\mu_X(x'),
\]
where $f$ runs over all $1$-Lipschitz functions on $X$.
If $\lambda(t) = t^2$, then $\Var_\lambda(f)$ is the usual variance of $f$.
The $\lambda$-variance $\Var_\lambda(\nu)$ of a Borel probability measure $\nu$ on $\R$
is defined by
\[
\Var_\lambda(\nu) := \int_\R \int_\R \lambda(|x-x'|)\;d\nu(x)d\nu(x').
\]

Denote by $\cV$ the set of Borel probability measures on $\R$
absolutely continuous with respect to the one-dimensional Lebesgue measure $\cL^1$
and with connected support,
and by $\cV_\lambda$ the set of $\nu \in \cV$ with finite $\lambda$-variance.
Note that $\cV_\lambda = \cV$ for bounded $\lambda$.

An mm-space $X$ is said to be \emph{essentially connected}
if we have $\mu_X^+(A)>0$ for any closed set $A\subset X$ with $0<\mu_X(A)<1$.

One of our main theorems in this paper is stated as follows.

\begin{thm} \label{thm:isop-rigidity}
  Let $X$ be an essentially connected geodesic mm-space with
  fully supported Borel probability measure.
  Assume that $X$ satisfies $\IC(\nu)$ for a measure $\nu \in \mathcal{V}_\lambda$.
  Then we have
  \[
  \ObsVar_\lambda(X) \le \Var_\lambda(\nu).
  \]
  The equality holds only if we have one of the following {\rm(1)}, {\rm(2)}, and {\rm(3)}.
  \begin{enumerate}
  \item $X$ is covered by minimal geodesics joining two fixed points
  $p$ and $q$ in $X$ with $d_X(p,q) = \diam X$.
  It is homeomorphic to a suspension provided $X$ is non-branching.
\item $X$ is covered by rays emanating from a fixed point in $X$.
  It is homeomorphic to a cone provided $X$ is non-branching.
\item $X$ is covered by straight lines in $X$ that may cross each other
  only on their branch points.
  It is homeomorphic to $Y \times \R$ for a metric space $Y$
  provided $X$ is non-branching.
\end{enumerate}
\end{thm}

Applying the theorem to a complete Riemannian manifold yields
the following.

\begin{cor} \label{cor:isop-rigidity}
  Let $X$ be a complete and connected Riemannian manifold
  with a fully supported Borel probability measure $\mu_X$.
  Assume that $(X,\mu_X)$ satisfies $\IC(\nu)$
  for a measure $\nu \in \mathcal{V}_\lambda$.
  Then we have
  \[
  \ObsVar_\lambda(X) \le \Var_\lambda(\nu).
  \]
  The equality holds only if $X$ is diffeomorphic to
  either a twisted sphere or $Y \times \R$ for a differentiable manifold $Y$.
\end{cor}

%また，Theorem \ref{thm:isop-rigidity}において，
%$I_X \circ V \ge V'$ a.e., $\ObsVar_\lambda(X) = \Var_\lambda(\nu)$ならば，
%実は$I_X \circ V = V'$ a.e. となる．しかし，$I_X \circ V = V'$ a.e. だからと言って，
%$\ObsVar_\lambda(X) = \Var_\lambda(\nu)$となるとは限らない．

%Cavalletti-Mondino(Thm 1.4)において，
%リッチ曲率が正のlower boundをもち（$RCD^*$），$I_X = V' \circ V^{-1}$
%のときに hard rigidity が証明されている．
%これと比べると，我々の結果は仮定と結論が弱い．
%?? Cavalletti-Mondino では一点で$I_X = V' \circ V^{-1}$
%を仮定．これから$I_X = V' \circ V^{-1}$が出るか？

A typical example of Theorem \ref{thm:isop-rigidity} and Corollary \ref{cor:isop-rigidity}
is obtained as a warped product manifold $(J\times F,dt^2 + \varphi(t)^2 g)$,
where $J$ is an interval of $\R$ and $(F,g)$ a compact Riemannian manifold
(see Section \ref{ssec:warped-prod} for the detail).

We remark that, in Theorem \ref{thm:isop-rigidity} and Corollary \ref{cor:isop-rigidity},
the equality assumption for the $\lambda$-observable variance cannot be
replaced by the existence of a straight line to obtain a topological splitting of $X$.
Such a counter example is shown in Section \ref{ssec:non-splitting}.

The isoperimetric comparison condition is much weaker than
the lower boundedness of Ricci curvature,
or the curvature-dimension condition due to
Lott-Villani \cite{LV:Ric-mm} and Sturm \cite{St:geomI, St:geomII}.
In fact, if an mm-space has positive Cheeger constant,
then it satisfies $\IC(\nu)$ for some measure $\nu \in \cV$
(see Proposition  \ref{prop:Cheeger}).
In particular, any essentially connected and compact Riemannian space
with cone-like singularities
satisfies $\IC(\nu)$ for some $\nu \in \cV$, however, it does not satisfy
the curvature-dimension condition in general.
Actually, we find no example of an essentially connected mm-space
that does not satisfy $\IC(\nu)$ for any $\nu$.

We obtain the equality $I_X \circ V = V'$ a.e.~on $V^{-1}(\Image\mu_X)$
from the assumption of Theorem \ref{thm:isop-rigidity}.
However, the equality $I_X \circ V = V'$ a.e. is strictly weaker than
$\ObsVar_\lambda(X) = \Var_\lambda(\nu)$ even under $\IC(\nu)$.
In fact, we prove that an mm-space with some mild condition
always satisfies $I_X \circ V = V'$ a.e.~for some $\nu$
(see Proposition \ref{prop:Cheeger} and Corollary \ref{cor:Cheeger}).

In the proof of Corollary \ref{cor:isop-rigidity},
we obtain an isoparametric function on $X$
as a $1$-Lipschitz function attaining the observable $\lambda$-variance.
Thus, the problem of whether the twisted sphere in Corollary \ref{cor:isop-rigidity}
is a sphere or not is related to a result of Qian-Tang \cite{QT:isopara}, in which
they proved that every odd-dimensional exotic sphere admits
no totally isoparametric function with two points as the focal set.
However, it seems to be difficult to prove that the isoparametric function in our proof is total.
Note that any twisted sphere of dimension at most six is diffeomorphic to
a sphere.

As an application of (the proof of) Theorem \ref{thm:isop-rigidity},
we obtain the following new splitting theorem.

\begin{thm} \label{thm:splitting}
  Let $X$ be a complete and connected Riemannian manifold
  with a fully supported smooth probability measure $\mu_X$
  of Bakry-\'Emery Ricci curvature bounded below by one.
  Assume that the one-dimensional Gaussian measure, say $\gamma^1$,
  on $\R$ has finite $\lambda$-variance.
  Then we have
  \[
  \ObsVar_\lambda(X) \le \Var_\lambda(\gamma^1)
  \]
  and the equality holds if and only if $X$ is isometric to $Y \times \R$
  and $\mu_X = \mu_Y \otimes \gamma^1$ up to an isometry,
  where $Y$ is a complete Riemannian manifold
  with a smooth probability measure $\mu_Y$
  of Bakry-\'Emery Ricci curvature bounded below by one.
\end{thm}

If $\lambda(t) = t^2$, then Theorem \ref{thm:splitting} follows from
Cheng-Zhou's result \cite{CZ:eigen} (see Remark \ref{rem:eigenObsVar}
for the detail).

We see some other famous splitting theorems for Bakry-\'Emery Ricci curvature
in the papers by Lichnerowicz \cite{Lch:splitting}
and Fang-Li-Zhang \cite{FLZ:splittting}.

Note that if the Bakry-\'Emery Ricci curvature is bounded away from zero,
then the total of the associated measure is always finite
(see \cite{Mg:density, St:geomI}), so that, for Theorem \ref{thm:splitting},
the assumption for the measure $\mu_X$ to be probability is not restrictive.

Although the assumption of Theorem \ref{thm:splitting} is stronger than
Corollary \ref{cor:isop-rigidity},
yet the existence of a straight line instead of the equality in Theorem \ref{thm:splitting}
is not enough for $X$ to split isometrically.
For instance, an $n$-dimensional hyperbolic plane
with a certain smooth probability measure
has Bakry-\'Emery Ricci curvature bounded below by one
(see \cite{WW:comp}*{Example 2.2}),
for which the equality in Theorem \ref{thm:splitting} does not hold.

It is a natural conjecture that Theorem \ref{thm:splitting} would be true also for
an $\RCD(1,\infty)$-space, for which we have no proof at present.
One of the difficulties is the lack of the first variation formula
of weighted area in an $\RCD$-space.

Considering the diameter, we have the following theorem.

\begin{thm} \label{thm:diam}
  Let $X$ be an essentially connected compact geodesic mm-space with
  a fully supported Borel probability measure.
  Assume that $X$ satisfies $\IC(\nu)$ for a measure $\nu \in \mathcal{V}$
  with compact support.
  Then we have
  \[
  \diam X \le \diam\supp\nu.
  \]
  The equality holds if and only if $\ObsVar_\lambda(X) = \Var_\lambda(\nu)$.
  Consequently, in the equality case, we have 
  {\rm(1)} of Theorem \ref{thm:isop-rigidity}.
\end{thm}

\begin{cor}
  Let $X$ be a complete and connected Riemannian manifold with
  a fully supported Borel probability measure.
  Assume that $X$ satisfies $\IC(\nu)$ for a measure $\nu \in \mathcal{V}$
  with compact support.
  Then we have
  \[
  \diam X \le \diam\supp\nu.
  \]
  The equality holds only if $X$ is diffeomorphic to a twisted sphere.
\end{cor}

Combining Theorem \ref{thm:diam} with
Ketterer's maximal diameter theorem \cite{Kt:cones}
and Cavalletti-Mondino's isoperimetric comparison theorem \cite{CvMd:sharp-isop},
we have the following.

\begin{cor}
  
  Let $X$ be an $\RCD^*(N-1,N)$-space and
  let $d\sigma^N(\theta) := C_N^{-1} \sin^{N-1}\theta\,d\theta$ on $[\,0,\pi\,]$,
  where $N > 1$ is a real number and $C_N := \int_0^\pi \sin^{N-1}\theta\,d\theta$.
  Then we have
  \[
  \ObsVar_\lambda(X) \le \Var_\lambda(\sigma^N),
  \]
  and the equality holds if and only if
  $X$ is isomorphic to the spherical suspension $Y \times_{\sin^{N-1}} [\,0,\pi\,]$ 
  over an $\RCD^*(N-2,N-1)$-space $Y$,
  where the spherical suspension over $Y$
  is equipped with the product measure $\mu_Y \otimes \sigma^N$.
\end{cor}

For $\lambda(t) = t^2$, we calculate the variance of $\sigma^N$ as follows:

\begin{equation} \label{eq:Var-sigma}
  \Var_{t^2}(\sigma^N)=\frac 12 \left(\zeta(2,h) -\sum_{k=0}^{\lceil\frac{N-1}2\rceil-1}
  \frac 1{(h+k)^2}\right)
\end{equation}
(see Section \ref{sec:Var-sigma}), where $\zeta(s,q) := \sum_{k=0}^\infty \frac{1}{(q+k)^s}$
is the Hurwitz zeta function, $h:=\frac{N-1}2-\lceil\frac{N-1}2\rceil+1\in(0,1]$,
and $\lceil x\rceil$ is the smallest integer not less than $x$.

\bigskip\noindent
{\bf Idea of proof of Theorem \ref{thm:isop-rigidity}.}
Let us show the idea of the proof of Theorem \ref{thm:isop-rigidity} briefly.
Theorem \ref{thm:isop-rigidity} follows from the two following theorems,
Theorems \ref{thm:isop} and \ref{thm:max-distr-simple}.

For two Borel probability measures $\mu$ and $\nu$ on $\R$,
we say that \emph{$\mu$ dominates $\nu$} if there exists
a $1$-Lipschitz function $f : \R \to \R$ such that $f_*\mu = \nu$,
where $f_*\mu$ is the push-forward of $\mu$ by $f$, often called the \emph{distribution} of $f$.
A Borel probability measure is called a \emph{dominant of $X$}
if it dominates $f_*\mu_X$ for any $1$-Lipschitz function $f : X \to \R$.

\begin{thm} \label{thm:isop}
  Let $X$ be an essentially connected geodesic mm-space.
  If $X$ satisfies $\IC(\nu)$ for a measure $\nu \in \cV_\lambda$,
  then $\nu$ is a dominant of $X$.
  In particular, we have
  \[
  \ObsVar_\lambda(X) \le \Var_\lambda(\nu).
  \]
\end{thm}

We prove a stronger version of this theorem in \S\ref{sec:IC-dom}
(see Theorem \ref{thm:isoTFAE}).
A weaker version of the theorem (see Corollary \ref{cor:iso-dominant})
was stated by Gromov \cite{Gmv:isop-waist}*{\S 9.1.B}
without proof.

By Theorem \ref{thm:isop}, we have the first part of Theorem \ref{thm:isop-rigidity}.
To prove the rigidity part, we assume $\IC(\nu)$ for $X$
and $\ObsVar_\lambda(X) = \Var_\lambda(\nu)$.
Then, we are able to find a $1$-Lipschitz function $f : X \to \R$
such that
\[
\Var_\lambda(f) = \ObsVar_\lambda(X) = \Var_\lambda(\nu).
\]
The push-forward measure $f_*\mu_X$ coincides with $\nu$ up to an isometry of $\R$.
Then Theorem \ref{thm:isop-rigidity} follows from the following.

\begin{thm} \label{thm:max-distr-simple}
  Let $X$ be a geodesic mm-space with fully supported probability measure.
  If there exists a $1$-Lipschitz function $f : X \to \R$ such that
  $f_*\mu_X$ is a dominant of $X$, then
  we have at least one of {\rm(1)}, {\rm(2)}, and {\rm(3)} of
  Theorem {\rm\ref{thm:isop-rigidity}}.
\end{thm}

In fact, if $f$ is bounded, then we have (1).
If only one of $\inf f$ and $\sup f$ is finite, then we have (2).
If both of $\inf f$ and $\sup f$ are infinite, then we have (3).
The minimal geodesic foliation in Theorem \ref{thm:max-distr-simple}
is generated by the gradient vector field of $f$ (in the smooth case),
where the gradient vector field of $f$ is a unit vector field.
In addition, under the assumption of Theorem \ref{thm:isop-rigidity},
the function $f$ becomes an isoparametric function, i.e.,
the Laplacian of $f$ is constant on each level set of $f$.

A more general and minute version of Theorem \ref{thm:max-distr-simple}
for any mm-space is proved in \S\ref{sec:max-distr} (see Theorem \ref{thm:max-distr}).
A primitive version of Theorem \ref{thm:max-distr-simple}
was obtain by the first named author \cite{Nkj:1-meas}.

\begin{ack}
  The authors would like to thank Prof.~Shouhei Honda
  for his valuable comments.
\end{ack}

\section{Preliminaries}

In this section, we see some basics on mm-spaces.
We refer to \cite{Gmv:green,Sy:mmg} for more details.

\begin{defn}[mm-Space]
  Let $(X,d_X)$ be a complete separable metric space
  and $\mu_X$ a Borel probability measure on $X$.
  We call the triple $(X,d_X,\mu_X)$ an \emph{mm-space}.
  We sometimes say that $X$ is an mm-space, in which case
  the metric and the measure of $X$ are respectively indicated by
  $d_X$ and $\mu_X$.
\end{defn}

\begin{defn}[mm-Isomorphism]
  Two mm-spaces $X$ and $Y$ are said to be \emph{mm-isomorphic}
  to each other if there exists an isometry $f : \supp\mu_X \to \supp\mu_Y$
  such that $f_*\mu_X = \mu_Y$,
  where $f_*\mu_X$ is the push-forward of $\mu_X$ by $f$.
  Such an isometry $f$ is called an \emph{mm-isomorphism}.
\end{defn}

Any mm-isomorphism between mm-spaces is automatically surjective,
even if we do not assume it.
The mm-isomorphism relation is an equivalent relation between mm-spaces.

Note that $X$ is mm-isomorphic to $(\supp\mu_X,d_X,\mu_X)$.
\emph{We assume that an mm-space $X$ satisfies
\[
X = \supp\mu_X
\]
unless otherwise stated.}

\begin{defn}[Lipschitz order] \label{defn:dom}
  Let $X$ and $Y$ be two mm-spaces.
  We say that $X$ (\emph{Lipschitz}) \emph{dominates} $Y$
  and write $Y \prec X$ if
  there exists a $1$-Lipschitz map $f : X \to Y$ satisfying
  \[
  f_*\mu_X = \mu_Y.
  \]
  We call the relation $\prec$ the \emph{Lipschitz order}.
\end{defn}

The Lipschitz order $\prec$ is a partial order relation
on the set of mm-isomorphism classes of mm-spaces.

%\subsection{Observable diameter}
%
%The observable diameter is one of the most fundamental invariants
%of an mm-space.
%
%\begin{defn}[Partial and observable diameter]
%  Let $X$ be an mm-space.
%  For a real number $\alpha$, we define
%  the \emph{partial diameter
%    $\diam(X;\alpha) = \diam(\mu_X;\alpha)$ of $X$}
%  to be the infimum of $\diam A$,
%  where $A \subset X$ runs over all Borel subsets
%  with $\mu_X(A) \ge \alpha$ and $\diam A$ denotes the diameter of $A$.
%  For a real number $\kappa > 0$, we define
%  the \emph{observable diameter of $X$} to be
%  \begin{align*}
%    \ObsDiam(X;-\kappa) &:= \sup\{\;\diam(f_*\mu_X;1-\kappa) \mid\\
%    &\qquad\qquad\text{$f : X \to \R$ is $1$-Lipschitz continuous}\;\}.
%  \end{align*}
%\end{defn}

%The observable diameter is invariant under mm-isomorphism.
%Note that $\ObsDiam_Y(X;-\kappa) = \diam(X;1-\kappa) = 0$ for $\kappa \ge 1$.

%\begin{defn}[L\'evy family]
%  A sequence of mm-spaces $X_n$, $n=1,2,\dots$,
%  is called a \emph{L\'evy family} if
%  \[
%  \lim_{n\to\infty} \ObsDiam(X_n;-\kappa) = 0
%  \]
%  for any $\kappa > 0$.
%\end{defn}
%
%\begin{prop} \label{prop:ObsDiam-dom}
%  If $X \prec Y$, then
%  \[
%  \ObsDiam(X;-\kappa) \le \ObsDiam(Y;-\kappa)
%  \]
%  for any $\kappa > 0$.
%\end{prop}

%\subsection{Separation distance}

\begin{defn}[Separation distance]
  Let $X$ be an mm-space.
  For any real numbers $\kappa_0,\kappa_1,\cdots,\kappa_N > 0$
  with $N\geq 1$,
  we define the \emph{separation distance}
  \[
  \Sep(X;\kappa_0,\kappa_1, \cdots, \kappa_N)
  = \Sep(\mu_X;\kappa_0,\kappa_1, \cdots, \kappa_N)
  \]
  of $X$ as the supremum of $\min_{i\neq j} d_X(A_i,A_j)$
  over all sequences of $N+1$ Borel subsets $A_0,A_2, \cdots, A_N \subset X$
  satisfying that $\mu_X(A_i) \geq \kappa_i$ for all $i=0,1,\cdots,N$,
  where $d_X(A_i,A_j) := \inf_{x\in A_i,y\in A_j} d_X(x,y)$.
  If $\kappa_i > 1$ for some $i$, then
  we define
  \[
  \Sep(X;\kappa_0,\kappa_1, \cdots, \kappa_N) 
  = \Sep(\mu_X;\kappa_0,\kappa_1, \cdots, \kappa_N)
  := 0.
  \]
\end{defn}

We see that $\Sep(X;\kappa_0,\kappa_1, \cdots, \kappa_N)$ is
monotone nonincreasing in each $\kappa_i$,
and that $\Sep(X;\kappa_0,\kappa_1, \cdots, \kappa_N) = 0$
if $\sum_{i=0}^N \kappa_i > 1$.

\begin{lem} \label{lem:Sep-prec}
  Let $X$ and $Y$ be two mm-spaces.
  If $X$ is dominated by $Y$, then we have
  \[
  \Sep(X;\kappa_0,\dots,\kappa_N) \le \Sep(Y;\kappa_0,\dots,\kappa_N)
  \]
  for any real numbers $\kappa_0,\dots,\kappa_N > 0$.
\end{lem}

\begin{defn} \label{defn:tpm}
  For a Borel probability measure on $\R$ and a real number $\alpha$,
  we define
  \begin{align*}
  t_+(\nu;\alpha) &:= \sup\{\; t \in \R \mid \nu([\,t,+\infty\,)) \ge \alpha\;\},\\
  t_-(\nu;\alpha) &:= \inf\{\; t \in \R \mid \nu((\,-\infty,t\,]) \ge \alpha\;\}.
\end{align*}
\end{defn}

We see that
$\nu([\,t_+(\nu;\alpha),+\infty\,)) \ge \alpha$ and
$\nu((\,-\infty,t_-(\nu;\alpha)\,]) \ge \alpha$.
For any $\kappa_0, \kappa_1 > 0$ with $\kappa_0 + \kappa_1 \le 1$, 
we have
\[
\Sep(\nu;\kappa_0,\kappa_1) = t_+(\nu;\kappa_0) - t_-(\nu;\kappa_1).
\]

% \begin{proof}
%   Assume that $\supp_X = X$ and $\supp_Y = Y$.
%   If we assume $X \prec Y$, then there is a $1$-Lipschitz map
%   $f : Y \to X$ such that $f_*\mu_Y = \mu_X$.
%   We take any Borel subsets $A_0,A_1,\dots,A_N \subset X$
%   such that $\mu_X(A_i) \ge \kappa_i$ for any $i$.
%   If we have no such sequence of Borel subsets, the the lemma is trivial.
%   Since $f$ is $1$-Lipschitz continuous and since
%   $\mu_Y(f^{-1}(A_i)) = \mu_X(A_i) \ge \kappa_i$, we have
%   \[
%   \min_{i\neq j} d_X(A_i,A_j)
%   \le \min_{i\neq j} d_Y(f^{-1}(A_i),f^{-1}(A_j))
%   \le \Sep(Y;\kappa_0,\dots,\kappa_N).
%   \]
%   This completes the proof.
% \end{proof}

%\begin{prop} \label{prop:ObsDiam-Sep}
%  For any mm-space $X$ and any real numbers $\kappa$ and $\kappa'$
%  with $\kappa > \kappa' > 0$, we have
%  \begin{align}
%    \tag{1} &\ObsDiam(X;-2\kappa) \le \Sep(X;\kappa,\kappa),\\
%    \tag{2} &\Sep(X;\kappa,\kappa) \le \ObsDiam(X;-\kappa').
%  \end{align}
%\end{prop}

%
\section{Isoperimetric Comparison and Domination of Measures} \label{sec:IC-dom}

Let $X$ be an mm-space and $\nu$ a Borel probability measure on $\R$.

\begin{defn}[Isoperimetric comparison condition of L{\'e}vy type]
  We say that $X$ satisfies the \emph{isoperimetric comparison condition of L{\'e}vy type}
  $\ICL(\nu)$
  if for any real numbers $a, b \in \supp\nu$ with $a\leq b$ and
  for any Borel set $A\subset X$ with $\mu_X(A) > 0$ we have
  \[
  V(a)\leq\mu_X(A)\Rightarrow V(b)\leq\mu_X(B_{b-a}(A)),
  \]
  where $V$ is the cumulative distribution function of $\nu$.
\end{defn}

\begin{rem}\label{rem:isoLevy}
  In the definition of $\ICL(\nu)$, the condition is equivalent
  if we restrict $A$ to be any closed set in $X$ with $\mu_X(A) > 0$.
\end{rem}

Recall that a dominant of $X$ is a Borel probability measure on $\R$
that dominates the distribution of any $1$-Lipschitz function on $X$.

\begin{defn}[Iso-dominant]
  A Borel probability measure $\nu$ is called an \emph{iso-dominant of $X$}
  if for any $1$-Lipschitz function $f : X \to \R$ there exists
  a monotone nondecreasing function $h : X \to \R$
  such that $f_*\mu_X = h_*\nu$.
\end{defn}

Any iso-dominant of $X$ is a dominant of $X$.

The purpose of this section is to prove the following theorem,
which is stronger than Theorem \ref{thm:isop}.

\begin{thm}\label{thm:isoTFAE}
  Let $X$ be an essentially connected geodesic mm-space
  and let $\nu \in \cV$.  Then the following {\rm(1)}, {\rm(2)}, and {\rm(3)} are equivalent to each other.
\begin{enumerate}
\item $\nu$ is an iso-dominant of $X$.
\item $X$ satisfies $\ICL(\nu)$.
\item $X$ satisfies $\IC(\nu)$.
\end{enumerate}
\end{thm}

We need several statements for the proof of Theorem \ref{thm:isoTFAE}.

\begin{prop}\label{prop:dominateToProfile}
  Let $X$ and $Y$ be mm-spaces
  such that $X$ dominates $Y$.
  Then we have
  \[
  \Image\mu_Y\subset \Image\mu_X
  \quad\text{and}\quad
  I_X \leq I_Y \ \text{on $\Image\mu_Y$}.
  \]
  In particular, if $X$ satisfies $\IC(\nu)$ for a Borel probability measure $\nu$ on $\R$,
  then $Y$ also satisfies $\IC(\nu)$.
\end{prop}
\begin{proof}
  Since $X$ dominates $Y$, there is a 1-Lipschitz map $f : X \to Y$
  such that $f_*\mu_X=\mu_Y$.
  For any Borel set $A\subset Y$, we see
  $f^{-1}(B_\varepsilon(A)) \supset B_\varepsilon(f^{-1}(A))$
  by the $1$-Lipschitz continuity of $f$, and so
  \begin{align*}
    \mu_Y^+(A) &= \limsup_{\varepsilon\to+0}
    \frac{ \mu_Y(B_\varepsilon(A))-\mu_Y(A)}\varepsilon\\
    &\geq \limsup_{\varepsilon\to+0}
    \frac{ \mu_X(B_\varepsilon(f^{-1}(A)))-\mu_X(f^{-1}(A)) }\varepsilon\\
    &=\mu_X^+(f^{-1}(A)),
  \end{align*}
  which implies that, for any $v\in\Image\mu_Y$,
  \begin{align*}
    I_Y(v)&=\inf_{\mu_Y(A)=v}\mu_Y^+(A)
    \geq\inf_{\mu_X(f^{-1}(A))=v}\mu_X^+(f^{-1}(A)) \ge I_X(v).
  \end{align*}
  The rest is easy.
  This completes the proof.
\end{proof}

Using Proposition \ref{prop:dominateToProfile} we prove the following.

\begin{prop}[Gromov {\cite[\S 9]{Gmv:isop-waist}}]\label{prop:dominantToProfile}
  If $\nu$ is a dominant of a geodesic mm-space $X$, then
  \[
  \Image\mu_X\subset\Image\nu
  \quad\text{and}\quad
  I_\nu \leq I_X \ \text{on $\Image\mu_X$},
  \]
  where $I_\nu$ is the isoperimetric profile of $(\R,\nu)$.
\end{prop}

\begin{proof}
  We take any real number $v \in \Image\mu_X$ and fix it.
  If $v = 0$, then it is obvious that $v \in \Image\nu$ and $I_\nu(v) = 0 = I_X(v)$.
  Assume $v > 0$.
  For any $\varepsilon>0$ there is a closed set $A\subset X$
  such that $\mu_X(A)=v$ and $\mu_X^+(A)<I_X(v)+\varepsilon$.
  Note that $A$ is nonempty because of $v > 0$.
  Define a function $f:X\to\R$ by
  \[
  f(x):=
  \begin{cases}
     d_X(x,A) & \text{if $x\in X \setminus A$},\\
     -d_X(x,X \setminus A) & \text{if $x\in A$}.
  \end{cases}
  \]
  Then $f$ is $1$-Lipschitz continuous.
  Since $f_*\mu_X((\,-\infty,0\,]) = \mu_X(A) = v$,
  we have
  \[
  I_{f_*\mu_X}(v)\leq (f_*\mu_X)^+((-\infty,0])=\mu_X^+(A)<I_X(v)+\varepsilon.
  \]
  Since $\nu$ dominates $f_*\mu_X$, 
  Proposition \ref{prop:dominateToProfile} implies that $v\in\Image\nu$
  and $I_\nu(v)\leq I_{f_*\mu_X}(v)$.
  We therefore have $I_\nu(v)<I_X(v)+\varepsilon$.
  By the arbitrariness of $\varepsilon>0$, we obtain $I_\nu(v)\leq I_X(v)$.
  This completes the proof.
\end{proof}

\begin{prop}\label{prop:dominantToLevy}
  Let $X$ be a geodesic mm-space and $\nu$ a Borel probability measure on $\R$.
  If $\nu$ is an iso-dominant of $X$, then $X$ satisfies $\ICL(\nu)$.
\end{prop}

\begin{proof}
  Assume that $\nu$ is an iso-dominant of $X$.
  We take any real numbers $a,b \in \supp\nu$ with $a \leq b$
  and any nonempty closed set $A \subset X$
  in such a way that $V(a) \leq \mu_X(A)$,
  where $V$ is the cumulative distribution function of $\nu$.
  Define a function $f : X \to \R$ by
  \[
  f(x):=\begin{cases}
    d_X(x,A) & \text{if $x\in X \setminus A$},\\
    -d_X(x,X \setminus A) & \text{if $x\in A$}
  \end{cases}
  \]
  for $x \in X$.
  Since $\nu$ is an iso-dominant of $X$,
  there is a monotone nondecreasing $1$-Lipschitz function
  $g : \R \to \R$ such that
  \[
  f_*\mu_X=g_*\nu
  \]
  We set
  \[
  a' := \sup g^{-1}((\,-\infty,0\,])
  \quad\text{and}\quad
  b'=\sup g^{-1}((\,-\infty,b-a\,]).
  \]
  The continuity and monotonicity of $g$ implies that
  \[
  g_*\nu((-\infty,0])=V(a')
  \quad\text{and}\quad
  g_*\nu((-\infty,b-a])=V(b'),
  \]
  which are true even if $a'$ and/or $b'$ are infinity.
  Since
  \[
  V(a) \leq \mu_X(A) = f_*\mu_X((\,-\infty,0\,]) = g_*\nu((\,-\infty,0\,]) = V(a')
  \]
  we have $a \leq a'$.
  By the monotonicity and the $1$-Lipschitz continuity of $g$,
  \[
  g(b) \leq g(a'+b-a) \leq g(a')+b-a \leq b-a,
  \]
  which implies $b\leq b'$ and therefore,
  \begin{align*}
    V(b)&\leq V(b')
    =g_*\nu((-\infty,b-a])\\
    &=f_*\mu_X((-\infty,b-a]) = \mu_X(B_{b-a}(A)).
  \end{align*}
  This completes the proof.
\end{proof}

\begin{prop}\label{prop:LevyToProfile}
  Let $X$ be an mm-space
  and $\nu$ a Borel probability measure on $\R$.
  If $X$ satisfies $\ICL(\nu)$, then $X$ satisfies $\IC(\nu)$.
\end{prop}

\begin{proof}
Assume $\ICL(\nu)$ for $X$.
It suffices to prove
\begin{equation} \label{eq:IC}
I_X \circ V(t) \ge V'(t)
\end{equation}
for $\cL^1$-a.e.~$t\in V^{-1}(\Image\mu_X)$.
We note that $V'(t)$ exists for $\cL^1$-a.e.~$t\in V^{-1}(\Image\mu_X)$.
If $t$ is not contained in $\supp\nu$, then \eqref{eq:IC} is clear because of $V'(t) = 0$.
Assume $t \in \supp\nu$.
If $(\,t,t+\varepsilon_0\,)$ does not intersect $\supp\nu$ for some $\varepsilon_0 > 0$,
then $V'(t) = 0$ if any, and we have \eqref{eq:IC}.
If otherwise, there is a sequence of positive real numbers $\varepsilon_i\to 0$
such that $t+\varepsilon_i$ is contained in $\supp\nu$.
Applying $\ICL(\nu)$ yields that $\mu_X(B_{\epsilon_i}(A)) \ge V(t+\varepsilon_i)$
for any Borel set $A \subset X$ with $\mu_X(A) = V(t)$.
We therefore have
\begin{align*}
I_X\circ V(t)&=\inf_{\mu_X(A)=V(t)}\mu_X^+(A)\\
&\ge \inf_{\mu_X(A)=V(t)}\limsup_{i \to \infty}
\frac{\mu_X(B_{\varepsilon_i}(A))-\mu_X(A)}{\varepsilon_i}\\
&\ge \lim_{i\to \infty}\frac{V(t+\varepsilon_i)-V(t)}{\varepsilon_i},
\end{align*}
which is equal to $V'(t)$ if any.
This completes the proof.
\end{proof}

For a monotone nondecreasing and right-continuous function
$F : \R \to [\,0,1\,]$ with $\lim_{t\to -\infty}F(t)=0$,
we define a function $\tilde F : [\,0,1\,] \to \R$ by
\[
\tilde F(s) :=
\begin{cases}
  \inf \{\;t \in \R \mid s\leq F(t) \;\} & \text{if $s \in (\,0,1\,]$},\\
  c & \text{if $s=0$}
\end{cases}
\]
for $s \in [\,0,1\,]$, where $c$ is a constant.

\begin{lem}\label{lem:tilde}
  For any $F$ as above, we have the following {\rm(1)}, {\rm(2)}, and {\rm(3)}.
  \begin{enumerate}
  \item\label{tilde_geq} $F\circ\tilde F(s)\geq s$ for any real number $s$ with $0 \le s 
  \le 1$.
  \item\label{tilde_leq} $\tilde{F} \circ F(t) \leq t$ for any real number $t$ with $F(t)>0$.
  \item\label{tilde_set} $F^{-1}((\,-\infty,t\,]) \setminus \{0\}=(\,0,F(t)\,]$ for any real number $t$.
  \end{enumerate}
\end{lem}

The proof of the lemma is straight forward and omitted (see \cite{Nkj:1-meas}).

\begin{lem}\label{lem:meas_decomp}
  Let $\mu$ be a Borel probability measure on $\R$
  with cumulative distribution function $F$.
  Then we have
  \[
  \mu=\tilde F_*\Leb{[\,0,1\,]},
  \]
  where $\Leb{[\,0,1\,]}$ is the one-dimensional Lebesgue measure on $[\,0,1\,]$.
\end{lem}

\begin{proof}
  For any $t>0$ we have, by Lemma \ref{lem:tilde}\eqref{tilde_set},
  \begin{align*}
    \tilde F_*\Leb{[\,0,1\,]}((\,-\infty,t\,])
    &=\Leb{[\,0,1\,]}(\tilde F^{-1}((\,-\infty,t\,])\setminus \{0\})\\
    &=\Leb{[\,0,1\,]}((\,0,F(t)\,])\\
    &=F(t)=\mu((\,-\infty,t\,]).
  \end{align*}
  This completes the proof.
\end{proof}

\begin{lem}\label{lem:meas_unif}
  Let $\mu$ be a Borel probability measure with
  cumulative distribution function $F$.
  If $F$ is continuous, then we have
  \[
  F_*\mu=\Leb{[\,0,1\,]}.
  \]
\end{lem}

\begin{proof}
  Let $s$ be any real number with $0 < s \le 1$.
  It follows from the definition of $\tilde{F}$ that
  $F(\tilde F(s)-\varepsilon) < s$ for any $\varepsilon > 0$.
  By the continuity of $F$, we have $F\circ\tilde F(s) \le s$,
  which together with Lemma \ref{lem:tilde}\eqref{tilde_geq}
  implies $F\circ\tilde F|_{(0,1]}=\mathrm{id}_{(0,1]}$.
  
  By Lemma \ref{lem:meas_decomp},
  \[
  F_*\mu=F_*\tilde F_*\Leb{[\,0,1\,]} = (F\circ\tilde F|_{(\,0,1\,]})_*\Leb{(\,0,1\,]} =
  (\mathrm{id}_{(\,0,1\,]})_*\Leb{(\,0,1\,]} = \Leb{[\,0,1\,]}.
  \]
  This completes the proof.
\end{proof}

Using Lemmas \ref{lem:meas_decomp} and \ref{lem:meas_unif}
we prove the following.

\begin{thm}\label{thm:LevyToDominant}
  Let $X$ be an mm-space and $\nu$ a Borel probability measure on $\R$
  with cumulative distribution function $V$.
  If $V$ is continuous and if $X$ satisfies $\ICL(\nu)$,
  then $\nu$ is an iso-dominant of $X$.
\end{thm}

\begin{proof}
  Let $f:X\to\R$ be a $1$-Lipschitz function.
  Denote by $F$ the cumulative distribution function of $f_*\mu_X$.
  We set $t_0:=\inf\supp\nu$.
  If $t_0=-\infty$, then we define $G :=\tilde F\circ V : \supp\nu\to\R$.
  If $t_0>-\infty$, then we define
  \[
  G(t):=
  \begin{cases}
  \tilde F\circ V(t) & \quad \text{if $t\neq t_0$},\\
  \lim_{s\to t_0}\tilde F\circ V(s) & \quad \text{if $t=t_0$}
  \end{cases}
  \]
  for $t \in \supp\nu$.
  We later prove the $1$-Lipschitz continuity of $\tilde F\circ V$ on $\supp\nu\setminus\{t_0\}$,
  which ensures the existence of the above limit.
  By Lemmas \ref{lem:meas_unif} and \ref{lem:meas_decomp},
  \[
  G_*\nu=\tilde F_*V_*\nu=\tilde F_*\Leb{[\,0,1\,]}=f_*\mu_X.
  \]
  The rest of the proof is to show the $1$-Lipschitz continuity of $G$.
  Since $V$ is monotone nondecreasing and so is $\tilde{F}$ on $(\,0,1\,]$,
  we see that $G$ is monotone nondecreasing on $\supp\nu\setminus\{t_0\}$.
  We take any two real numbers $a$ and $b$ with $t_0<a\leq b$.
  It suffices to prove that $G(b)\leq G(a)+b-a$.
  By Lemma \ref{lem:tilde}\eqref{tilde_geq},
  \begin{align*}
    V(a)&\leq(F\circ\tilde F)(V(a))\\
    &=F\circ G(a)\\
    &=\mu_X(f^{-1}((-\infty,G(a)])).
  \end{align*}
  We remark that the $\mu_X$-measure of $f^{-1}((\,-\infty,G(a)\,])$ is nonzero
  because of $V(a)>0$.
  By $\ICL(\nu)$,
  \begin{align*}
    V(b)&\leq\mu_X(B_{b-a}(f^{-1}((-\infty,G(a)])))\\
    &\leq\mu_X(f^{-1}(B_{b-a}((-\infty,G(a)])))\\
    &=f_*\mu_X((-\infty,G(a)+b-a])\\
    &=F(G(a)+b-a),
  \end{align*}
  which together with the monotonicity of $\tilde{F}$ on $(\,0,1\,]$
  and with Lemma \ref{lem:tilde}\eqref{tilde_leq} proves
  \begin{align*}
    G(b)&=(\tilde F\circ V)(b)\\
    &\leq\tilde F\circ F(G(a)+b-a)\\
    &\leq G(a)+b-a
  \end{align*}
  This completes the proof.
\end{proof}

\begin{lem}\label{lem:integral}
  Let $g:\R\to\R$ be a monotone nondecreasing function,
  $f:\R\to[\,0,+\infty\,)$ a Borel measurable function,
  and $A\subset\R$ a Borel set.
  Then we have
  \[
  \int_{g^{-1}(A)}(f\circ g)\cdot g' \; d\mathcal L^1\leq \int_A f \; d\mathcal L^1.
  \]
\end{lem}

\begin{proof}
  Let us first prove that
  \begin{equation} \label{eq:integral}
     \int_{g^{-1}(A)}g'\mathcal L^1\leq\mathcal L^1(A)
  \end{equation}
  for any Borel set $A \subset \R$.
  Let $I$ be an open interval in $\R$.
  For a natural number $n$, we set
  \begin{align*}
    a_n:=
    \begin{cases}
    \inf g^{-1}(I)+\frac 1n \quad &\text{if $\inf g^{-1}(I)>-\infty$},\\
    -n\quad &\text{if $\inf g^{-1}(I)=-\infty$},
    \end{cases}\\
    b_n:=
    \begin{cases}
    \inf g^{-1}(I)-\frac 1n \quad &\text{if $\inf g^{-1}(I)<\infty$},\\
    n\quad &\text{if $\inf g^{-1}(I)=\infty$}.
    \end{cases}
  \end{align*}
  $\{a_n\}$ is monotone decreasing
  and $\{b_n\}$ monotone increasing.
  For every sufficiently large $n$, we have $a_n\leq b_n$ and $a_n,b_n\in g^{-1}(I)$.
  We also see that $\lim_{n\to\infty} a_n = \inf g^{-1}(I)$
  and $\lim_{n\to\infty} b_n = \sup g^{-1}(I)$.
  Since
  \[
  \int_{[a_n,b_n]}g' \; d\mathcal L^1 \leq g(b_n)-g(a_n) \leq \sup I-\inf I=\mathcal L^1(I),
  \]
  Lebesgue's monotone convergence theorem proves
  \[
  \int_{g^{-1}(I)}g' \; d\mathcal L^1=\int_{(\,\inf g^{-1},\sup g^{-1}\,)} g' \; d\mathcal L^1
  \leq \mathcal L^1(I).
  \]
  Since any open set in $\R$ is the union of countably many mutually disjoint open intervals,
  we have \eqref{eq:integral} for any open set in $\R$.
  By the outer regularity of $\cL^1$,
  any Borel set $A \subset \R$ can be approximated by an open set containing $A$
  and therefore we have \eqref{eq:integral} for any Borel set in $\R$.
  
  Approximating $f$ by a simple function and applying \eqref{eq:integral},
  we obtain the lemma.
\end{proof}

\begin{thm}\label{thm:profileToLevy}
  Let $X$ be an essentially connected mm-space and
  $\nu \in \cV$.
  If $X$ satisfies $\IC(\nu)$, then $X$ satisfies $\ICL(\nu)$.
\end{thm}

\begin{proof}
  Setting $E:=(\supp\nu)^\circ$,
  we easily see the bijectivity of $V|_E : E \to (\,0,1\,)$.
  We define a function $\rho : \R \to \R$ by
  \[
  \rho(t):=
  \begin{cases}
    V'(t) & \text{for any $t \in V^{-1}(\Image\mu_X)$ where $V$ is differentiable}\\
    & \text{and such that $I_X\circ V(t)\geq V'(t)$},\\
    0& \text{otherwise},
  \end{cases}
  \]
  for a real number $t$.
  We see that $\rho = V'$ $\mathcal L^1$-a.e.
  and that $\rho$ is a density function of $\nu$ with respect to $\cL^1$.
  Since $I_X\circ V \geq \rho$ everywhere on $V^{-1}(\Image\mu_X)$,
  we have $I_X \geq \rho \circ(V|_E)^{-1}$
  on $\Image\mu_X\setminus\{0,1\}$.
  To prove $\ICL(\nu)$, we take two real numbers $a,b\in\supp\nu$ with $a\leq b$
  and a nonempty Borel set $A\subset X$ with $V(a)\leq\mu_X(A)$.
  We may assume $\mu_X(B_{b-a}(A)) < 1$.
  Let $s$ be any real number with $0 \le s \le b-a$.
  Remarking $\mu_X(B_s(A))\in\Image\mu_X\setminus\{0,1\}$, we see
  \[
  \mu_X^+(B_s(A))\geq I_X(\mu_X(B_s(A))) \geq
  \rho\circ(V|_E)^{-1}(\mu_X(B_s(A))).
  \]
  Setting $g(s):=\mu_X(B_s(A))$, we have
  \[
  g'(s)=\mu_X^+(B_s(A)) \ge \rho\circ(V|_E)^{-1}(g(s))
  \quad \text{$\cL^1$-a.e. $s \ge 0$}
  \]
  and so
  \[
  1 \le \frac{g'(s)}{\rho\circ(V|_E)^{-1}(g(s))} \le +\infty
  \quad \text{$\mathcal L^1$-a.e $s \in [\,0,+\infty\,)$},
  \]
  where we remark that $g'(s) > 0$ because of the essential connectivity of $X$.
  Since $g(0)=\mu_X(\bar A)\geq\mu_X(A)$, we have
  \begin{align*}
    (V|_E)^{-1}\circ g(0)&\geq(V|_E)^{-1}(\mu_X(A))\\
    &\geq(V|_E)^{-1}(V(a))=a,
  \end{align*}
  so that, by Lemmas \ref{lem:integral} and \ref{lem:meas_decomp},
  \begin{align*}
    b-a
    &\leq \int_{[\,0,b-a\,]}g'(s)\cdot \left(\rho\circ(V|_E)^{-1}(g(s))\right)^{-1}ds\\
    &\leq \int_{g^{-1}(g([\,0,b-a\,]))}g'(s)\cdot \left(\rho\circ(V|_E)^{-1}(g(s))\right)^{-1}ds\\
    &\leq \int_{g([\,0,b-a\,])}\frac{d\mathcal L^1}{\rho\circ(V|_E)^{-1}}\\
    &= \int_{(V|_E)^{-1}\circ g([\,0,b-a\,])}\frac{1}{\rho} \; d((V|_E)^{-1}_*\mathcal L^1)\\
    &= \int_{(V|_E)^{-1}\circ g([\,0,b-a\,])}\frac{1}{\rho} \; d\nu\\
    &\le \int_{(V|_E)^{-1}\circ g([\,0,b-a\,])}d\mathcal L^1\\
    &\leq \mathcal L^1([(V|_E)^{-1}\circ g(0),(V|_E)^{-1}\circ g(b-a)])\\
    &= (V|_E)^{-1}\circ g(b-a)-(V|_E)^{-1}\circ g(0)\\
    &\leq (V|_E)^{-1}\circ g(b-a)-a,
  \end{align*}
  which implies
  \[
  V(b)\leq g(b-a)=\mu_X(B_{b-a}(A)).
  \]
  This completes the proof.
\end{proof}

\begin{proof}[Proof of Theorem \ref{thm:isoTFAE}]
  The theorem follows from
  Propositions \ref{prop:dominantToLevy}, \ref{prop:LevyToProfile},
  Theorems \ref{thm:profileToLevy} and \ref{thm:LevyToDominant}.
\end{proof}

\begin{defn}[Iso-simpleness]
  A Borel probability measure $\nu$ on $\R$ is said to be \emph{iso-simple}
  if $\nu \in \cV$ and if
  \[
  I_\nu\circ V=V' \quad \text{$\mathcal L^1$-a.e.}
  \]
\end{defn}

\begin{rem}
  For any Borel probability measure $\nu$ on $\R$,
  we always observe $I_\nu\circ V\leq V'$ $\mathcal L^1$-a.e.
  In fact, we have
  \[
  V'(t)=\nu^+((-\infty,t])\geq \inf_{\nu(A)=V(t)} \nu^+(A) = I_\nu\circ V(t) = I_\nu\circ V(t)
  \]
  $\cL^1$-a.e. $t$.
\end{rem}

In the case where $\nu$ is iso-simple,
$\IC(\nu)$ is equivalent to $I_\nu \le I_X$.
This together with Theorem \ref{thm:isoTFAE} and Proposition \ref{prop:dominantToProfile}
implies the following corollary.

\begin{cor}[Gromov \cite{Gmv:isop-waist}*{\S 9}] \label{cor:iso-dominant}
  Let $X$ be an essentially connected mm-space
  and $\nu$ an iso-simple Borel probability measure on $\R$.
  Then, we have $I_\nu \le I_X$ if and only if $\nu$ is an iso-dominant of $X$.
\end{cor}

Gromov \cite{Gmv:isop-waist}*{\S 9} stated this corollary without proof.

\section{Maximum Distribution of $1$-Lipschitz Function} \label{sec:max-distr}

The purpose of this section is to prove the following theorem,
which is a generalization and also a refinement of Theorem \ref{thm:max-distr-simple}.
A geodesic is said to be \emph{normal} if its metric derivative is one everywhere.

\begin{thm} \label{thm:max-distr}
Let $X$ be an mm-space with fully supported probability measure $\mu_X$
such that $X$ is embedded in a geodesic metric space $\tilde{X}$ isometrically.
Assume that the distribution $f_*\mu_X$ of a $1$-Lipschitz function $f : X \to \R$
is a dominant of $X$.
Then we have the following {\rm(1)}, {\rm(2)}, and {\rm(3)}.
\begin{enumerate}
\item If $\inf f > -\infty$ and if $\sup f < +\infty$, then
\begin{enumerate}
\item[(1-a)] there exist a unique minimizer of $f$, say $p$,
and a unique maximizer of $f$, say $q${\rm;}
\item[(1-b)] $X$ is covered by minimal geodesics joining $p$ and $q$ in $\tilde{X}${\rm;}
\item[(1-c)] for any point $x \in X$ we have
\[
f(x) = d_X(p,x) + f(p) = -d_X(q,x) + f(q).
\]
\end{enumerate}
\item If $\inf f > -\infty$ and if $\sup f = +\infty$, then
\begin{enumerate}
\item[(2-a)] there exists a unique minimizer of $f$, say $p${\rm;}
\item[(2-b)] for any real number $L > 0$ and any point $x \in X$,
there exists a minimal normal geodesic in $\tilde{X}$ emanating from $p$ passing through $x$
and with length not less than $L${\rm;}
\item[(2-c)] for any point $x \in X$ we have
\[
f(x) = d_X(p,x) + f(p).
\]
\end{enumerate}
\item If $\inf f = -\infty$ and if $\sup f = +\infty$, then
\begin{enumerate}
\item[(3-a)] there exists a $1$-Lipschitz extension $\tilde{f} : \tilde{X} \to \R$ of $f$
such that for any $L > 0$ and any $x \in X$ there exists a minimal normal geodesic 
$\gamma : [\,-L,L\,] \to \tilde{X}$ with $\gamma(0) = x$ such that
$\tilde{f}(\gamma(t)) = f(x) + t$ for any $t \in [\,-L,L\,]${\rm;}
\item[(3-b)] for any $a \in \R$ and $x \in X$ we have
\[
f(x) =
\begin{cases}
  -d(x,\tilde{f}^{-1}(a)) + a &\text{if $f(x) < a$},\\
  d(x,\tilde{f}^{-1}(a)) + a &\text{if $f(x) \ge a$}.
\end{cases}
\]
\end{enumerate}
\end{enumerate}
\end{thm}

Since any metric space can be embedded into a Banach space
by the Kuratowski embedding, for any given $X$ the space $\tilde{X}$
as in Theorem \ref{thm:max-distr}
always exists.

For the proof of Theorem \ref{thm:max-distr}, we need several lemmas.
From now on, let $\tilde{X}$, $X$, and $f : X \to \R$ be
as in Theorem \ref{thm:max-distr}.
We first prove the following.

\begin{lem} \label{lem:tail}
Let $g : X \to \R$ be a $1$-Lipschitz function satisfying
the following conditions {\rm(i)--(iv)}.
\begin{enumerate}
\item[(i)] If $\inf f > -\infty$, then $\inf f \ge \inf g$.
\item[(ii)] If $\inf f = -\infty$, then there exists a real number $\alpha$
such that $f_*\mu_X = g_*\mu_X$ on $(\,-\infty,\alpha\,]$.
\item[(iii)] If $\sup f < +\infty$, then $\sup f \le \sup g$.
\item[(iv)] If $\sup f = +\infty$, then there exists a real number $\beta$
such that $f_*\mu_X = g_*\mu_X$ on $[\,\beta,+\infty\,)$.
\end{enumerate}
Then, the two measures $f_*\mu_X$ and $g_*\mu_X$
coincides with each other up to an isometry of $\R$.
\end{lem}

\begin{proof}
Since $f_*\mu_X$ dominates $g_*\mu_X$,
there is a $1$-Lipschitz map $h : \R \to \R$ such that
$h_*f_*\mu_X = g_*\mu_X$.
We put $a := \inf f$, $b := \sup f$, $a' := \inf g$, and $b' := \inf g$.

If $a > -\infty$ and if $b < +\infty$, then
we have $(\,a,b\,) \subset (\,a',b'\,)$ by (i) and (iii).
Since $h$ maps $(\,a,b\,)$ to $(\,a',b'\,)$ and by the $1$-Lipschitz continuity,
we obtain $(\,a,b\,) = (\,a',b'\,)$ and $h$ is an isometry from $[\,a,b\,]$
to itself.  We have the lemma in this case.

Assume that $a > -\infty$ and $b = +\infty$.
Then, by (i) and (iv), we have $a' \le a$ and $b' = +\infty$.
Since $h$ maps $\supp f_*\mu_X$ to $\supp g_*\mu_X$,
we have
\[
h([\,a,+\infty\,)) \supset
\begin{cases}
 [\,a',+\infty\,) &\text{if $a' > -\infty$},\\
 \R &\text{if $a' = -\infty$}.
\end{cases}
\]
Let
\[
a'' :=
\begin{cases}
 a' &\text{if $a' > -\infty$},\\
 a - 1 &\text{if $a' = -\infty$}.
\end{cases}
\]
There is a number $t_0$ such that $t_0 \ge a$ and $h(t_0) = a''$.
It follows from the $1$-Lipschitz continuity of $h$ that
\begin{equation} \label{eq:tail}
  h(t) \le t + a'' - t_0 \le t + a - t_0 \le t
\end{equation}
for any $t \ge t_0$.
For the $\beta$ as in (iv), we set $\beta_0 := \max\{\beta,t_0\}$.
Let $\lambda : \R \to (\,0,1\,)$ be a strictly monotone decreasing
continuous function.
Since $h^{-1}([\,\beta_0,+\infty\,)) \supset [\,\beta_0,+\infty)$, we see that
\begin{align*}
\int_{[\,\beta_0,+\infty\,)} \lambda \; d(f_*\mu_X)
&= \int_{[\,\beta_0,+\infty\,)} \lambda \; d(g_*\mu_X)
= \int_{[\,\beta_0,+\infty\,)} \lambda \; d(h_*f_*\mu_X)\\
&= \int_{h^{-1}([\,\beta_0,+\infty\,))} \lambda \circ h \; d(f_*\mu_X)\\
&\ge \int_{[\,\beta_0,+\infty\,)} \lambda \circ h \; d(f_*\mu_X)\\
&\ge \int_{[\,\beta_0,+\infty\,)} \lambda \; d(f_*\mu_X),
\end{align*}
which implies that $h(t) = t$ for any $t \ge \beta_0$.
This together with \eqref{eq:tail} proves that
$t_0 = a = a' = a''$ and $h(t) = t$ for any $t \ge a$.
The lemma follows in this case.

If $a = -\infty$ and if $b < +\infty$, then we obtain the lemma
in the same way as above.

We assume that $a = -\infty$ and $b = +\infty$.
For $0 < \kappa < 1$, we set
\begin{align*}
A_-(\kappa) &:= (\,-\infty,t_-(g_*\mu_X;\kappa/2)\,],\\
A_+(\kappa) &:= [\,t_+(g_*\mu_X;\kappa/2),+\infty\,).
\end{align*}
where $t_\pm(\dots)$ is as in Definition \ref{defn:tpm}.
We have
\[
f_*\mu_X(h^{-1}(A_\pm(\kappa))) = h_*f_*\mu_X(A_\pm(\kappa))
= g_*\mu_X(A_\pm(\kappa)) \ge \kappa/2,
\]
which together with the $1$-Lipschitz continuity of $h$ proves
\begin{align} \label{eq:h}
\Sep(g_*\mu_X;\kappa/2,\kappa/2)
&= t_+(g_*\mu_X;\kappa/2) - t_-(g_*\mu_X;\kappa/2)\\
&= d_\R(A_-(\kappa),A_+(\kappa)) \notag\\
&\le d_\R(h^{-1}(A_-(\kappa)),h^{-1}(A_+(\kappa))) \notag\\
&\le \Sep(f_*\mu_X;\kappa/2,\kappa/2) \notag\\
&= t_+(f_*\mu_X;\kappa/2) - t_-(f_*\mu_X;\kappa/2). \notag
\end{align}
By (ii) and (iv), if $\kappa$ is small enough, then
\[
t_\pm(f_*\mu_X;\kappa/2) = t_\pm(g_*\mu_X;\kappa/2) =: t_\pm(\kappa/2),
\]
which implies the equalities of \eqref{eq:h}.
Therefore, the interval between
$A_-(\kappa)$ and $A_+(\kappa)$
and the interval between
$h^{-1}(A_-(\kappa))$ and $h^{-1}(A_+(\kappa))$
both coincide with $[\,t_-(\kappa/2),t_+(\kappa/2)\,]$.
The $h$ maps $[\,t_-(\kappa/2),t_+(\kappa/2)\,]$ to itself isometrically.
Since we have $t_\pm(\kappa/2) \to \pm\infty$ as $\kappa \to 0+$,
the map $h$ is an isometry of $\R$.
This completes the proof of Lemma \ref{lem:tail}.
\end{proof}

\begin{defn}[Generalized signed distance function]
Let $S$ be a metric space.
A function $g : S \to \R$ is called a \emph{generalized signed distance function}
if there exist three mutually disjoint subsets
$\Omega_+$, $\Omega_0$, and $\Omega_-$ of $S$
and a real number $a$ such that
\begin{enumerate}
\item[(i)] $\Omega_+$ and $\Omega_-$ are open sets and $\Omega_0$ is a closed set;
\item[(ii)] $S = \Omega_+ \cup \Omega_0 \cup \Omega_-$ and
$\partial\Omega_+ \cup \partial\Omega_- \subset \Omega_0$;
\item[(iii)] for any $x \in S$,
\[
g(x) =
\begin{cases}
  d_S(x,\Omega_0) + a &\text{if $x \in \Omega_+$},\\
  a &\text{if $x \in \Omega_0$},\\
  -d_S(x,\Omega_0) + a &\text{if $x \in \Omega_-$}.
\end{cases}
\]
\end{enumerate}
\end{defn}

Any generalized signed distance function $g$ on a geodesic space $S$
is $1$-Lipschitz continuous and has the property that
\[
d_X(g^{-1}(a),g^{-1}(b)) = |\,a-b\,|
\]
for any $a,b \in g(S)$.

\begin{lem} \label{lem:exc}
Let $A$, $B$, and $\Omega$ be three subsets of $\tilde{X}$
such that $A$ and $B$ are both closed, $d_{\tilde{X}}(A,B) > 0$,
and $A \cup B \subset \Omega$.
We take two real numbers $a$ and $b$ in such a way that  $d_{\tilde{X}}(A,B) = b - a$.
Assume that there exists a point $x_0 \in X \cap \Omega \setminus (A \cup B)$
such that
\[
d_{\tilde{X}}(x_0,A) + d_{\tilde{X}}(x_0,B) > d_{\tilde{X}}(A,B).
\]
Then, there exist a real number $c \in (\,a,b\,)$
and a family $\{h_t : \Omega \to \R\}_{t \in (\,-r_0,r_0\,)}$
of $1$-Lipschitz functions, $r_0 > 0$,
such that, for any $t \in (\,-r_0,r_0\,)$, we have
$h_t = a$ on $A$, $h_t = b$ on $B$, $c + t \in [\,a,b\,]$,
and $c+t$ is an atom of $(h_t)_*\mu_X$. 
\end{lem}

\begin{proof}
Setting
\[
\delta := \frac{1}{2}( d_{\tilde{X}}(x_0,A) + d_{\tilde{X}}(x_0,B) - d_{\tilde{X}}(A,B) ),
\]
we have $\delta > 0$ by the assumption.

(i) In the case where $d_{\tilde{X}}(x_0,A), d_{\tilde{X}}(x_0,B) > \delta$,
we define
\[
r_A := d_{\tilde{X}}(x_0,A) - \delta, \ \ r_B := d_{\tilde{X}}(x_0,B) - \delta,
\ \ r_0 := \min\{\delta,r_A,r_B\}.
\]
We then see that
\begin{align}
 r_A + r_B &= d_{\tilde{X}}(A,B), \label{eq:AB}\\
 r_0 &\le \min\{r_A,r_B\}, \label{eq:r0}\\
 r_A &\le d_{\tilde{X}}(x_0,A) - r_0, \label{eq:r0A}\\
 r_B &\le d_{\tilde{X}}(x_0,B) - r_0 \label{eq:r0B}.
\end{align}

(ii) In the case where $d_{\tilde{X}}(x_0,A) \le \delta$ or $d_{\tilde{X}}(x_0,B) \le \delta$,
we have only one of
$d_{\tilde{X}}(x_0,A) \le \delta$ and $d_{\tilde{X}}(x_0,B) \le \delta$
because of the definition of $\delta$.
Without loss of generality, we may assume that
$d_{\tilde{X}}(x_0,A) \le \delta$.
Define
\begin{gather*}
  r_A := \frac12 \min\{d_{\tilde{X}}(x_0,A),d_{\tilde{X}}(A,B)\},\\
  r_B := d_{\tilde{X}}(A,B) - r_A, \qquad
  r_0 := \min\{r_A,r_B\}.
\end{gather*}
Then we immediately obtain
\eqref{eq:AB}, \eqref{eq:r0}, \eqref{eq:r0A}.
By $d_{\tilde{X}}(x_0,A) \le \delta$, we have
$d_{\tilde{X}}(x_0,B) \ge d_{\tilde{X}}(x_0,A) + d_{\tilde{X}}(A,B)$,
which proves \eqref{eq:r0B}.

In either of the cases (i) or (ii),
we define
$c := r_A + a$ and
\[
h_t(x) :=
\begin{cases}
  d_{\tilde{X}}(x,A) + a &\text{if $d_{\tilde{X}}(x,A) \le r_A + t$,}\\
  - d_{\tilde{X}}(x,B) + b &\text{if $d_{\tilde{X}}(x,B) \le r_B - t$,}\\
  c+t &\text{otherwise},
\end{cases}
\]
for $t \in (\,-r_0,r_0\,)$ and $x \in \Omega$.
Then $h_t$ is $1$-Lipschitz continuous on $\Omega$.

It follows from \eqref{eq:r0A} that, for any $t \in (\,-r_0,r_0\,)$,
the distance between any point in $U_{r_0 - |t|}(x_0)$ and $A$
is greater than $r_A+|t|$.
In the same way, from \eqref{eq:r0B},
the distance between any point in $U_{r_0 - |t|}(x_0)$ and $B$
is greater than $r_B+|t|$.
We therefore have
$U_{r_0 - |t|}(x_0) \subset h_t^{-1}(c+t)$
and so
\[
(h_t)_*\mu_X(\{c+t\}) \ge \mu_X(U_{r_0 - |t|}(x_0)) > 0,
\]
because of $x_0 \in X = \supp\mu_X$.
The family of the functions $h_t$, $t \in (\,-r_0,r_0\,)$, satisfies all the claims of the lemma.
This completes the proof of Lemma \ref{lem:exc}.
\end{proof}

\begin{lem} \label{lem:signed-dist}
Let $g : \tilde{X} \to \R$ be a generalized signed distance function
that is an extension of $f$.
For any point $x \in X$ and two real numbers $a$ and $b$ with
$\inf g \le a < f(x) < b \le \sup g$, we have
\[
d_{\tilde{X}}(x,g^{-1}(a)) + d_{\tilde{X}}(x,g^{-1}(b)) = b - a.
\]
\end{lem}

\begin{proof}
Since $d_{\tilde{X}}(g^{-1}(a),g^{-1}(b)) = b - a$,
a triangle inequality proves
\[
d_{\tilde{X}}(x,g^{-1}(a)) + d_{\tilde{X}}(x,g^{-1}(b)) \ge b - a.
\]
Suppose that there are $x_0$, $a$, $b$ such that
\[
d_{\tilde{X}}(x_0,g^{-1}(a)) + d_{\tilde{X}}(x_0,g^{-1}(b)) > b - a.
\]
We apply Lemma \ref{lem:exc} for $\Omega := g^{-1}([\,a,b\,])$,
$A := g^{-1}(a)$, and $B := g^{-1}(b)$
to obtain a family of $1$-Lipschitz functions
$h_t : g^{-1}([\,a,b\,]) \to \R$, $t \in (\,-r_0,r_0\,)$,
as in Lemma \ref{lem:exc}.
We extend $h_t$ to a function on $\tilde{X}$ by setting
$h_t := g$ on $g^{-1}((\,-\infty,a\,) \cup (\,b,+\infty\,))$.
Then $h_t$ is $1$-Lipschitz continuous on $\tilde{X}$
and $c + t$ is an atom of $(h_t)_*\mu_X$ for any $t \in (\,-r_0,r_0\,)$.
It follows from Lemma \ref{lem:tail} that
$(h_t)_*\mu_X$ and $f_*\mu_X$ coincide with each other
up to an isometry of $\R$.
As a result, $f_*\mu_X$ has uncountably many atoms,
which is a contradiction because $f_*\mu_X$ is a probability measure.
This completes the proof.
\end{proof}

From now on, translating $f$ if necessary, we assume that $f$ has $0$ as an median.
For a $1$-Lipschitz extension $\hat{f} : \tilde{X} \to \R$ of $f$,
we define a generalized signed distance function $\tilde{f} : \tilde{X} \to \R$ by
\begin{equation} \label{eq:tilde-f}
  \tilde{f}(x) :=
  \begin{cases}
    - d_{\tilde{X}}(x,\hat{f}^{-1}(0)) & \text{if $x \in \hat{f}^{-1}((\,-\infty,0\,))$},\\
    d_{\tilde{X}}(x,\hat{f}^{-1}(0)) & \text{if $x \in \hat{f}^{-1}([\,0,+\infty,))$}.
  \end{cases}
\end{equation}
for $x \in \tilde{X}$.
It holds that $f(x)$ and $\tilde{f}(x)$ have the same sign for any $x \in X$
and that $|f| \le |\tilde{f}|$ on $X$ by the $1$-Lipschitz continuity of $f$.

\begin{lem} \label{lem:ext}
  We have $\tilde{f} = f$ on $X$.
\end{lem}

\begin{proof}
For $0 \le \alpha \le 1$, we set
\[
t_-(\alpha) := t_-(f_*\mu_X;\alpha) \quad\text{and}\quad
t_+(\alpha) := t_+(f_*\mu_X;\alpha).
\]
Note that $t_-(1/2)$ is the minimum of medians of $f$
and $t_+(1/2)$ is the maximum of medians of $f$.
Since $f$ has $0$ as an median, we have
$t_-(1/2) \le 0 \le t_+(1/2)$.

Let us first prove $f_*\mu_X = \tilde{f}_*\mu_X$.
Let $\kappa$ be any real number with $0 < \kappa \le 1$.
We see that
\begin{align*}
  \Sep(f_*\mu_X; \kappa/2,\kappa/2) &= t_+(\kappa/2) - t_-(\kappa/2),\\
  \Sep(\tilde{f}_*\mu_X; \kappa/2,\kappa/2)
  &= t_+(\tilde{f}_*\mu_X;\kappa/2) - t_-(\tilde{f}_*\mu_X;\kappa/2).
\end{align*}
It follows from $|f| \le |\tilde{f}|$ that
$\mu_X(f\le t) \le \mu_X(\tilde{f} \le t)$ for any $t \le 0$ and
$\mu_X(f \ge t) \le \mu_X(\tilde{f} \ge t)$ for any $t \ge 0$.
Since $t_-(\kappa/2) \le 0$, $t_+(\kappa/2) \ge 0$,
we have
$\mu_X(\tilde{f} \le t_-(\kappa/2)) \ge \mu_X(f \le t_-(\kappa/2)) \ge \kappa/2$
and
$\mu_X(\tilde{f} \ge t_+(\kappa/2)) \ge \mu_X(f \ge t_+(\kappa/2)) \ge \kappa/2$.
Therefore,
\[
t_+(\tilde{f}_*\mu_X;\kappa/2) \ge t_+(\kappa/2)
\quad\text{and}\quad
t_-(\tilde{f}_*\mu_X;\kappa/2) \le t_-(\kappa/2).
\]
Since $f_*\mu_X$ dominates $\tilde{f}_*\mu_X$, we see
\[
\Sep(\tilde{f}_*\mu_X;\kappa/2,\kappa/2) \le \Sep(f_*\mu_X;\kappa/2,\kappa/2).
\]
We thus obtain
\[
t_+(\tilde{f}_*\mu_X;\alpha) = t_+(\alpha) \quad\text{and}\quad
t_-(\tilde{f}_*\mu_X;\alpha) = t_-(\alpha)
\]
for any $\alpha \in (\,0,1/2\,]$, which yields
$f_*\mu_X = \tilde{f}_*\mu_X$.

Suppose that there is a point $x_0 \in X$ such that $f(x_0) \neq \tilde{f}(x_0)$.
Then we have $f(x_0) \neq 0$, because $\tilde{f}(x_0) = 0$ if $f(x_0) = 0$.

Assume that $0 < f(x_0) \neq \tilde{f}(x_0)$.
We have $f(x_0) < \tilde{f}(x_0)$.  Setting
\[
r_0 := \frac{\tilde{f}(x_0) - f(x_0)}{2} \quad\text{and}\quad
t_0 := \frac{f(x_0) + \tilde{f}(x_0)}{2},
\]
we have $\tilde{f} > \tilde{f}(x_0) - r_0 = t_0$ and
$f < f(x_0) + r_0 = t_0$ on $U_{r_0}(x_0)$,
which implies $\mu_X(\tilde{f} \ge t_0) > \mu_X(f \ge t_0)$.
This is a contradiction.

In the case where $0 > f(x_0) \neq \tilde{f}(x_0)$, we are led to a contradiction
in the same way.
We thus obtain $f = \tilde{f}$ on $X$.
This completes the proof of Lemma \ref{lem:ext}.
\end{proof}

\begin{lem} \label{lem:minimizer}
\begin{enumerate}
\item If $\inf f > -\infty$, then $f$ has a unique minimizer.
\item If $\sup f < +\infty$, then $f$ has a unique maximizer.
\end{enumerate}
\end{lem}

\begin{proof}
(2) follows from applying (1) for $-f$.

We prove (1).
Let us first prove the existence of a minimizer of $f$.
We find a sequence of points $x_n \in X$, $n=1,2,\dots$,
in such a way that $f(x_n)$ converges to $\inf f$ as $n\to\infty$.
If $\{x_n\}$ has a convergent subsequence, then its limit is a minimizer.
Suppose that $\{x_n\}$ has no convergent subsequence.
Replacing it by a subsequence, we assume that
$d_X(x_m,x_n) \ge 2\delta > 0$ and $f(x_n) < \inf f + \delta/2$
for any natural numbers $m\neq n$
and for a real number $\delta > 0$.
Define
$b := \inf f + \delta$, $r_n := d_{\tilde{X}}(x_n,\tilde{f}^{-1}(b))$, and
\[
g_n(x) :=
\begin{cases}
  d_X(x_n,x) - r_n + b &\text{if $x \in B_{r_n}(x_n)$},\\
  f(x) &\text{if $f(x) \ge b$},\\
  b &\text{otherwise}.
\end{cases}
\]
for $x \in \tilde{X}$.
The function $g_n$ is $1$-Lipschitz continuous on $\tilde{X}$.
It follows from 
Lemma \ref{lem:ext} that $r_n = b - f(x_n)$ and so
$\delta/2 \le r_n \le \delta$.
Therefore, $B_{\delta/2}(x_1)$ and $B_{r_n}(x_n)$ for any $n \ge 2$
are disjoint to each other.
Since $f \le f(x_1) + \delta/2 < b$ on $B_{\delta/2}(x_1)$,
we have $g_n = b$ on $B_{\delta/2}(x_1)$ for $n \ge 2$,
which implies
\begin{equation} \label{eq:gn}
  (g_n)_*\mu_X(\{b\}) \ge \mu_X(B_{\delta/2}(x_1)) + f_*\mu_X(\{b\}),
  \quad n \ge 2.
\end{equation}
Since $(g_n)_*\mu_X$ is dominated by $f_*\mu_X$,
there is a $1$-Lipschitz map $h_n : \R \to \R$ such that
$(h_n)_*f_*\mu_X = (g_n)_*\mu_X$.
Let 
$\varepsilon_n := \inf g_n - \inf f$.
Since $\inf g_n = g_n(x_n) = b - r_n$, we see that
$\varepsilon_n = \delta - r_n \ge 0$ and $\varepsilon_n \to 0$ as $n\to\infty$.
It follows from $g_n = f$ on $\{f \ge b\}$ that
$h_n(t) = t$ for any $t \ge b$.

We now prove that
\begin{equation} \label{eq:hn}
  h_n^{-1}(b) \cap \supp f_*\mu_X \subset [\,b - \epsilon_n,b\,]
\end{equation}
in the following.
In fact, $h_n$ does not increase but could decrease the distance between two points.
However, since $h_n([\,\inf f, b\,]) \supset [\,\inf g_n,b\,]$,
the function $h_n$ decreases the distance between two points
not more than $\epsilon_n$.
In particular, if a real number $t \in \supp f_*\mu_X$ satisfies $t < b - \epsilon_n$,
then $h_n(t) \neq b$.
This implies \eqref {eq:hn}.

By \eqref{eq:hn} and \eqref{eq:gn},
\begin{align*}
  f_*\mu_X([\,b-\epsilon_n,b)) &= f_*\mu_X([\,b-\varepsilon_n,b\,]) - f_*\mu_X(\{b\})\\
  &\ge f_*\mu_X(h_n^{-1}(b)) - f_*\mu_X(\{b\})\\
  &= (g_n)_*\mu_X(\{b\}) - f_*\mu_X(\{b\})\\
  &\ge \mu_X(B_{\delta/2}(x_1)) > 0,
\end{align*}
which is a contradiction.
The function $f$ has a minimizer.

We next prove the uniqueness of the minimizer of $f$.
Suppose that $f$ has two different minimizers $p$ and $q$.
We take a real number $b$ with $\inf f < b < \sup f$.
Define $r := b-\inf f$ and
\[
g(x) :=
\begin{cases}
  d_X(p,x) + \inf f &\text{if $x \in B_r(p)$},\\
  f(x) &\text{if $f(x) \ge b$},\\
  b &\text{otherwise}
\end{cases}
\]
for $x \in \tilde{X}$.
The function $g$ is $1$-Lipschitz continuous on $\tilde{X}$.
By $\inf g = \inf f$, Lemma \ref{lem:tail} implies $g_*\mu_X = f_*\mu_X$.
However, in the same discussion as in \eqref{eq:gn},
we obtain
\[
g_*\mu_X(\{b\}) > f_*\mu_X(\{b\}),
\]
which is a contradiction.
This completes the proof of Lemma \ref{lem:minimizer}.
\end{proof}

\begin{proof}[Proof of Theorem {\rm\ref{thm:max-distr}}]
Without loss of generality, it may be assumed that
$f$ has $0$ as an median.
By Lemma \ref{lem:ext}, the function $\tilde{f}$ defined in \eqref{eq:tilde-f}
is a $1$-Lipschitz extension of $f$.

We prove (1).
By Lemma \ref{lem:minimizer},
the function $f$ has a unique minimizer $p \in X$
and a unique maximizer $q \in X$.
Applying Lemma \ref{lem:signed-dist}
for
$g := \tilde{f}$, $a := f(p)$, $b := f(q)$ yields
\[
d_X(p,x) + d_X(x,q) = d_X(p,q) = b-a
\]
for any $x \in X$, which together with the $1$-Lipschitz continuity of $f$
leads us to (1).

We prove (2).
By Lemma \ref{lem:minimizer}, the function $f$ has a unique minimizer $p \in X$.
Applying Lemma \ref{lem:signed-dist} for $g := \tilde{f}$, $a := f(p)$
yields that, for $L > f(x)$,
\[
d_X(p,x) + d_X(x,\tilde{f}^{-1}(L)) = d_X(p,\tilde{f}^{-1}(L)) = L - a,
\]
which together with 
the $1$-Lipschitz continuity of $f$
leads us to (2).

(3) is obtained by applying Lemma \ref{lem:signed-dist} for $g := \tilde{f}$.

This completes the proof of Theorem \ref{thm:max-distr}.
\end{proof}

\section{Proof of Main Theorem}

In this section, we prove Theorems \ref{thm:isop-rigidity}, \ref{thm:splitting},
and \ref{thm:diam}
by using Theorems \ref{thm:isoTFAE} and \ref{thm:max-distr}.

We need the following lemma.

\begin{lem} \label{lem:f}
  Let $\nu$ be a dominant of an mm-space $X$ such that
  \[
  \ObsVar_\lambda(X) = \Var_\lambda(\nu) < +\infty.
  \]
  Then, there exists a $1$-Lipschitz function $f : X \to \R$
  such that $f_*\mu_X = \nu$.
\end{lem}

\begin{proof}
  Let $x_0 \in X$ be a fixed point.
  There is a sequence of $1$-Lipschitz functions $f_n : X \to \R$ with $f_n(x_0) = 0$
  such that $\Var_\lambda(f_n)$ converges to $\ObsVar_\lambda(X) = \Var_\lambda(\nu)$
  as $n\to\infty$.
  By Lemma \cite{Sy:mmg}*{Lemma 4.45},
  there is a subsequence of $\{f_n\}$ that converges in measure
  to a $1$-Lipschitz function $f : X \to \R$.
  We denote the subsequence by the same notation $\{f_n\}$.
  It follows from \cite{Sy:mmg}*{Lemma 1.26} that $(f_n)_*\mu_X$
  converges weakly to $f_*\mu_X$ as $n\to\infty$.
  Since $\nu$ dominates $(f_n)_*\mu_X$,
  there is a $1$-Lipschitz function $h_n : \R \to \R$ for each $n$ such that
  $(h_n)_*\nu = (f_n)_*\mu_X$.
  Since $(h_n)_*\nu$ converges weakly and by the $1$-Lipschitz continuity of $h_n$,
  we have the boundedness of $\{h_n(t)\}$ for any fixed $t \in \R$.
  By the Arzel\`a-Ascoli theorem,
  there is a subsequence of $\{h_n\}$
  that converges uniformly on compact sets.
  We replace $\{n\}$ by such a subsequence.
  Since $\lambda(|h_n(x)-h_n(x')|) \le \lambda(|x-x'|)$ for any $x,x'\in\R$
  and $\Var_\lambda(\nu) < +\infty$,
  the Lebesgue dominated convergence theorem proves
  that $\Var_\lambda((h_n)_*\nu)$ converges to $\Var_\lambda(h_*\nu)$
  as $n\to\infty$.
  We therefore have $\Var_\lambda(h_*\nu) = \Var_\lambda(\nu)$,
  which together with the $1$-Lipschitz continuity of $h$
  implies that $h$ is an isometry on the support of $\nu$.
  Since $(f_n)_*\mu_X = (h_n)_*\nu$ converges weakly to
  $f_*\mu_X$ and also to $h_*\nu$,
  we obtain $f_*\mu_X = h_*\nu$.
  Let $\tilde{h} : \R \to \R$ be the isometric extension of $h|_{\supp\nu}$.
  The composition $f\circ \tilde{h}^{-1}$ is the desired $1$-Lipschitz function.
  This completes the proof.
\end{proof}

\begin{lem} \label{lem:f-gamma}
  Let $f : X \to \R$ be a $1$-Lipschitz function on an mm-space $X$
  such that $f_*\mu_X$ is a dominant of $X$,
  and let $\gamma : I \to X$ be a $1$-Lipschitz curve
  defined on an interval $I \subset \R$.
  If $f(\gamma(t)) = f(\gamma(t_0)) + t$ for any number $t \in I$ and for a number $t_0 \in I$,
  then $\gamma$ is a minimal normal geodesic.
\end{lem}

\begin{proof}
  The assumption and the $1$-Lipschitz continuity of $f$ and $\gamma$ together imply
  \[
  |s-t| = |f(\gamma(s))-f(\gamma(t))| \le d_X(\gamma(s),\gamma(t))
  \le |s-t|
  \]
  for any $s,t \in I$.
  This complete the proof.
\end{proof}

\begin{proof}[Proof of Theorem \ref{thm:isop-rigidity}]
  Let $X$ be an essentially connected geodesic mm-space with
  fully supported Borel probability measure
  such that $X$ satisfies $\IC(\nu)$ for a measure $\nu \in \mathcal{V}_\lambda$.
  Theorem \ref{thm:isoTFAE} implies
  that $\nu$ is an iso-dominant of $X$.
  We therefore have
  \[
  \ObsVar_\lambda(X) \le \Var_\lambda(\nu).
  \]
  We assume the equality of the above.
  By Lemma \ref{lem:f}, there is
  a $1$-Lipschitz function $f : X \to \R$
  such that $f_*\mu_X$ coincides with $\nu$ up to an isometry of $\R$.
  Applying Theorem \ref{thm:max-distr}
  for $X$ (= $\tilde{X}$) and $f$ yields one of (1), (2), and (3)
  of Theorem \ref{thm:max-distr}.

  In the case of (2), we prove that for any point $x \in X$
  there is a ray emanating from the minimizer $p$ of $f$ and passing through $x$.
  In fact, we have a minimal geodesic from $p$ to $x$, say $\gamma$.
  We extend $\gamma$ to a maximal one as a minimal geodesic from $p$.
  If $\gamma$ is not a ray,  then
  it extends beyond $x$ by (2-b), which is a contradiction to the maximality of $\gamma$.
  Thus, $X$ is covered by rays emanating from $p$.
  
  In the case of (3), the discussion using (3-a)
  proves that $X$ is covered by the a family of normal straight lines
  $\gamma_\lambda$, $\lambda \in \Lambda$,
  such that
  \begin{equation} \label{eq:f-gamma}
    f(\gamma_\lambda(t)) = f(\gamma_\lambda(0)) + t
  \end{equation}
  for any $t \in \R$ and $\lambda \in \Lambda$.
  Assume that $\gamma_\lambda$ and $\gamma_{\lambda'}$ have
  a crossing point $\gamma_\lambda(a) = \gamma_{\lambda'}(b)$.
  Let $\sigma(t) := \gamma_\lambda(t)$ for $t \le a$
  and $\sigma(t) := \gamma_{\lambda'}(t-a+b)$ for $t > a$.
  Then, $\sigma : \R \to X$ is a $1$-Lipschitz curve.
  It follows from \eqref{eq:f-gamma} that
  \[
  f(\sigma(t)) = f(\sigma(0)) + t
  \]
  for any $t \in \R$.
  Lemma \ref{lem:f-gamma} yields that $\sigma$ is a minimal normal straight line,
  i.e., the crossing point $\gamma_\lambda(a) = \gamma_{\lambda'}(b)$
  is a branch point of $\gamma_\lambda$ and $\gamma_{\lambda'}$.
  This completes the proof of Theorem \ref{thm:isop-rigidity}.
\end{proof}

For the proof of the splitting theorem, we need the following lemma.

\begin{lem} \label{lem:splitting}
Let $X$ be a complete and connected Riemannian manifold
with a fully supported smooth probability measure $\mu_X$
and let $\nu \in \cV_\lambda$,
where $\lambda : [\,0,+\infty\,) \to [\,0,+\infty\,)$ is a strictly monotone
increasing continuous function.
If $X$ satisfies $\IC(\nu)$ and if $\Var_\lambda(f) = \Var_\lambda(\nu)$
for a $1$-Lipschitz function $f : X \to \R$, then
$f$ is a $C^\infty$ isoparametric function satisfying
$|\nabla f| = 1$ everywhere on $X$.
\end{lem}

\begin{proof}
Theorem \ref{thm:isoTFAE} tells us that
the distribution of $f$ coincides with $\nu$ up to an isometry of $\R$
and is an iso-dominant of $X$.
By Theorem \ref{thm:max-distr},
we have $U_\varepsilon(f^{-1}(\,-\infty,a\,])) = f^{-1}(\,-\infty,a+\varepsilon\,])$
for any $a \in \R$ and $\varepsilon > 0$,
so that the sublevel sets of $f$ realize the isoperimetric profile of $X$.
The first variation formula of weighted area
(see \cite{Mg:gmt}*{\S 18.9} and \cite{Gmv:isop-waist}*{\S 9.4.E})
proves that each level set of $f$ has constant weighted mean curvature
with respect to the weight $\mu_X$.
By the result of \cite{Alm},
each level set of $f$ is a hypersurface possibly with singularities.
However, by Theorem \ref{thm:max-distr}(3), the level sets of $f$ are all perpendicular to 
the minimal geodesics foliating $X$.
Thus, there are no singularity in the level sets of $f$ and also
no focal points to the level sets.
Therefore, $f$ is of $C^\infty$ and $|\nabla f| = 1$ everywhere on $X$.
As a result, $f$ turns out to be an isoparametric function on $X$.
\end{proof}

\begin{proof}[Proof of Theorem \ref{thm:splitting}]
We apply Theorem \ref{thm:isop-rigidity} for the one-dimensional
standard Gaussian measure $\gamma^1$ on $\R$ as $\nu$.
Let $f : X \to \R$ be a $1$-Lipschitz function attaining
the $\lambda$-observable variance of $X$.
By Lemma \ref{lem:splitting}, the function $f$ is a $C^\infty$
isoparametric function with $|\nabla f| = 1$ everywhere.
By translating $f$ if necessary, the distribution of $f$ coincides with $\gamma^1$.
The weighted area of $f^{-1}(t)$ with respect to $\mu_X$ is
\[
A(t) := \frac{1}{\sqrt{2\pi}} e^{-t^2/2}.
\]
We have $A'(t) = -t A(t)$.
Since the weighted mean curvature coincides with the drifted Laplacian of $f$,
we see $A'(t) = Lf(x) A(t)$ for $x \in f^{-1}(t)$,
where $L := \Delta - \nabla \psi$ is the drifted Laplacian on $X$
with respect to the weight function $e^{-\psi}$ of $X$.
We therefore have $Lf = -f$.
The Bochner-Weizenb\"ock formula
\[
L\frac{|\nabla f|^2}{2} - \langle\nabla f,\nabla L f\rangle
= \|\Hess f\|_{HS}^2 + \Ric_{\mu_X}(\nabla f,\nabla f)
\]
(see \cite{Vl:oldnew}*{the next to (14.46)})
leads us to $\Hess f = o$.
The manifold $X$ splits as $Y \times \R$ (see \cite{Inm:splitting}),
where $Y$ is a complete Riemannian manifold.
Let $d\mu_X(y,t) = e^{-\psi(y,t)} d\vol_Y(y) dt$ in the coordinate
$(y,t) \in Y \times \R$.
Since $\Ric(\nabla f,\nabla f) = 0$, we have
\[
1 = \Ric_{\mu_X}(\nabla f,\nabla f) = \Hess\psi(\nabla f,\nabla f)
= \frac{\partial^2}{\partial t^2} \psi(y,t),
\]
which implies $\psi(y,t) = \psi(y,0) + t^2/2$.
Defining the weight of $Y$ as
$d\mu_Y(y) := e^{-\psi(y,0)} \, d\vol_Y(y)$,
we obtain the theorem.
\end{proof}

\begin{rem} \label{rem:eigenObsVar}
  We see that
  the first nonzero eigenvalue (or the spectral gap) $\lambda_1$ of the drifted Laplacian
  on a complete Riemannian manifold $X$ with a full supported Borel probability measure
  satisfies
  \begin{equation} \label{eq:eigenObsVar}
  \lambda_1 \le \frac{1}{\ObsVar_{t^2}(X)}.
  \end{equation}
  In fact, since the energy of any $1$-Lipschitz function on $X$ is not greater than one,
  the Rayleigh quotient of any $1$-Lipschitz function is not greater than
  the inverse of the variance of it, which proves \eqref{eq:eigenObsVar}.

  Assume that a complete and connected Riemannian manifold $X$
  has Bakry-\'Emery Ricci curvature bounded below by one.
  In the case where $\ObsVar_{t^2}(X) = 1 \ (= \Var_{t^2}(\gamma^1))$,
  the inequality \eqref{eq:eigenObsVar} implies $\lambda_1 \le 1$.
  Thus, Theorem \ref{thm:splitting} for $\lambda(t) = t^2$ is also derived from the following.
  \begin{thm}[Cheng-Zhou \cite{CZ:eigen}]
    Let $X$ be a complete and connected Riemannian manifold
    with a fully supported smooth measure $\mu_X$ of Bakry-\'Emery
    Ricci curvature bounded below by one.
    Then, the drifted Laplacian has only discrete spectrum
    and we have
    \[
    \lambda_1 \ge 1.
    \]
    The equality holds only if $X$ is isometric to $Y \times \R$
  and $\mu_X = \mu_Y \otimes \gamma^1$ up to an isometry,
  where $Y$ is a complete Riemannian manifold
  with a smooth probability measure $\mu_Y$
  of Bakry-\'Emery Ricci curvature bounded below by one.
  \end{thm}

  If $\ObsVar_{t^2}(X) = 1$, then the function
  $f : X = Y \times \R \to \R$ defined by $f(y,t) = t$
  attains the observable variance of $X$ and also is
  an eigenfunction for $\lambda_1 = 1$.
\end{rem}

We prove Theorem \ref{thm:diam}.

\begin{proof}[Proof of Theorem \ref{thm:diam}]
  We assume that $X$ as in the theorem satisfies $\IC(\nu)$
  for a measure $\nu \in \cV$ with compact support.
  Theorem \ref{thm:isoTFAE} tells us that $\nu$ is a dominant of $X$.
  Let $\varphi : X \to \R$ be any $1$-Lipschitz function.
  Since $\varphi_*\mu_X$ is dominated by $\nu$,
  we have $\diam \varphi(X) = \diam\supp\varphi_*\mu_X \le \diam\supp\nu$.
  This implies $\diam X \le \diam\supp\nu$.
  
  Assume $\diam X = \diam \supp\nu$.
  By the compactness of $X$, there is a pair of points $p,q \in X$
  attaining the diameter of $X$.
  Letting $f := d_X(p,\cdot)$, we have
  $\diam\supp f_*\mu_X = \diam f(X) = \diam X = \diam \supp\nu$,
  which together with $f_*\mu_X \prec \nu$
  proves that $f_*\mu_X$ and $\nu$ coincide with each other
  up to an isometry of $\R$ and,
  in particular, $\ObsVar_\lambda(X) \ge \Var_\lambda(f) = \Var_\lambda(\nu)$.
  Since $\nu$ is a dominant of $X$,
  we obtain $\ObsVar_\lambda(X) = \Var_\lambda(\nu)$.
  
  Conversely, we assume $\ObsVar_\lambda(X) = \Var_\lambda(\nu)$.
  By Lemma \ref{lem:f}, we find a $1$-Lipschitz function $f : X \to \R$
  such that $f_*\mu_X = \nu$.
  We therefore have
  $\diam X \ge \diam f(X) = \diam\supp f_*\mu_X = \diam\supp\nu$,
  so that $\diam X = \diam\supp\nu$.
  This completes the proof.
\end{proof}
\section{Cheeger Constant and Isoperimetric Comparison Condition}

\begin{defn}[Cheeger constant] \label{defn:Cheeger-const}
  The \emph{Cheeger constant} $h(X)$ of an mm-space $X$ is defined by
  \[
  h(X) := 
  \inf_{0 < \mu_X(A) < 1} \frac{\mu_X^+(A)}{\min\{\mu_X(A),\mu_X(X \setminus A)\}}.
  \]
\end{defn}

The purpose of this section is to prove the following proposition,
which is useful to obtain an mm-space with the isoperimetric comparison
condition.  We have an application of this proposition in Section \ref{ssec:non-splitting}.

\begin{prop} \label{prop:Cheeger}
  Let $X$ be an mm-space with positive Cheeger constant.
  Then, $X$ is essentially connected and
  satisfies $\IC(\nu)$ for some measure $\nu \in \cV$.
  If, in addition, $I_X$ is Lebesgue measurable, then
  \[
  I_X \circ V = V'  \quad\text{$\cL^1$-a.e.}
  \]
  for some $\nu \in \cV$.
\end{prop}

We refer to \cite{Rt:conti}*{Section 1} for the descriptions for several works concerning
the regularity of the isoperimetric profile of a Riemannian manifold.
E. Milman \cite{EMlm:role}*{Lemma 6.9} proved
the $(n-1)/n$-H\"older continuity of the isoperimetric profile
of a complete and connected Riemannian manifold
with an absolutely continuous probability measure with respect to the volume measure
such that its density is bounded from above on every ball.
This together with Proposition \ref{prop:Cheeger}
implies the following.

\begin{cor} \label{cor:Cheeger}
  Let $X$ be a complete and connected Riemannian manifold
  and $\mu_X$ a fully supported probability measure on $X$ absolutely continuous
  with respect to the volume measure such that its density is bounded from above
  on every ball in $X$.
  Assume that $(X,\mu_X)$ has positive Cheeger constant.
  Then there exists a measure $\nu \in \cV$ such that
  \[
  I_X \circ V = V'  \quad\text{$\cL^1$-a.e.}
  \]
\end{cor}

For the proof of the proposition, we prove a lemma.

\begin{lem}\label{lem:construct}
  Let $\varphi : (\,0,1\,) \to [\,0,+\infty\,)$ be a Lebesgue measurable function
  such that $1/\varphi$ is locally integrable on $(\,0,1\,)$.
  Then, there exists a measure $\nu \in \cV$
  such that
  \[
  \varphi \circ V = V' \quad\text{$\cL^1$-a.e.},
  \]
  where $V$ is the cumulative distribution function of $\nu$.
\end{lem}

\begin{proof}
  Let $d\mu(t) := (1/\varphi(t)) \, dt$ on $(\,0,1\,)$.
  Note that $\varphi > 0$ $\cL^1$-a.e.
  By the assumption, $\mu$ is absolutely continuous with respect to $\cL^1$.
  We define a function $\rho : (\,0,1\,) \to \R$ by
  \[
  \rho(x):=\int_{\frac 12}^x\frac {dt}{\varphi(t)}
  \]
  for $x \in (\,0,1\,)$.
  Then, $\rho$ is a strictly monotone increasing and locally absolutely continuous
  function with connected image $\Image\rho$.
  We denote by $V : \Image\rho \to (\,0,1\,)$ the inverse function of $\rho$.
  The function $V$ is also strictly monotone increasing.
  Since $\lim_{t\to(\inf\Image\rho) +0}V(t)=0$ and $\lim_{t\to(\sup\Image\rho) -0}V(t)=1$,
  there exists a Borel probability measure $\nu$ on $\R$ possessing $V$
  as its cumulative distribution function.
  For any two real numbers $a$ and $b$ with $0 < a \le b < 1$, we see that
  \[
    V_*\cL^1([\,a,b\,]) = \cL^1(\rho([\,a,b\,])) = \rho(b) - \rho(a),
  \]
  so that $d(V_*\cL^1)(t) = (1/\varphi(t))\,dt$.
  This implies that
  \begin{align*}
    \int_a^b \varphi\circ V \, d\cL^1&=\int_{V([a,b])} \varphi \, d(V_*\cL^1)
    = \int_{V([\,a,b\,])} dt \\
    &= V(b)-V(a) = \nu([a,b]).
  \end{align*}
  Thus, $\nu$ is absolutely continuous with respect to $\cL^1$
  with density $\varphi \circ V$.
  Since $V'$ is also a version of the density of $\nu$, 
  we have $\varphi \circ V = V'$ $\cL^1$-a.e.
  This completes the proof.
\end{proof}

\begin{lem} \label{lem:Cheeger_varphi}
  Let $X$ be an mm-space with positive Cheeger constant.
  Then we have the following {\rm(1)}, {\rm(2)}, and {\rm(3)}.
  \begin{enumerate}
  \item $X$ is essentially connected.
  \item There exists a Lebesgue measurable function $\varphi : (\,0,1\,) \to (\,0,+\infty\,)$
    such that
    \begin{enumerate}
    \item $\varphi \le I_X$ on $\Image\mu_X$,
    \item $1/\varphi$ is locally integrable on $(\,0,1\,)$.
    \end{enumerate}
  \item If $I_X$ is Lebesgue measurable, then $1/I_X$ is locally integrable on $(\,0,1\,)$.
  \end{enumerate}
\end{lem}

\begin{proof}
  It follows from the definitions of $h(X)$ and $I_X(v)$ that
  \[
  h(X)\leq \frac{I_X(v)}{\min\{v,1-v\}}
  \]
  for any $v \in \Image \mu_X \setminus \{0,1\}$.
  Since $h(X) > 0$, we have $I_X(v) > 0$, which implies (1).
  Setting
  \[
  \varphi(v) := h(X)\,\min\{v,1-v\}
  \]
  for $v \in (\,0,1\,)$, we have (2).
  If $I_X$ is Lebesgue measurable, then the local integrability of $1/I_X$ on $(\,0,1\,)$
  follows from (2).
  This completes the proof.
\end{proof}

\begin{proof}[Proof of Proposition \ref{prop:Cheeger}]
  The proposition follows from Lemmas \ref{lem:construct} and \ref{lem:Cheeger_varphi}.
\end{proof}

\section{Examples}

\subsection{Warped product} \label{ssec:warped-prod}
We take a compact $n$-dimensional Riemannian manifold $F$ with Riemannian metric $g$
and a smooth function $\varphi : (\,a,b\,) \to (\,0,+\infty\,)$, $-\infty \le a < b \le +\infty$,
in such a way that $\int_a^b \varphi(t)^n \, dt = 1$.
Let $X$ be the completion of the warped product Riemannian manifold
$((\,a,b\,) \times M,dt^2 + \varphi(t)^2 g)$,
and $f : X \to \R$ the continuous extension of
the projection $(\,a,b\,) \times M \ni (t,x) \mapsto t \in \R$.
We equip $X$ with the probability measure
\[
d\mu_X(t,x) := dt \otimes \varphi(t)^n d\mu_F(x),
\]
where $\mu_F$ is a smooth probability measure on $F$ with full support.
Note that if the total volume of $F$ is one and if $\mu_F$ is taken to be the Riemannian volume
measure on $F$, then $\mu_X$ as defined above coincides with
the Riemannian volume measure.
We see that $f_*\mu_X = \varphi(t)^n\,dt$.

We consider the following.

\begin{asmp} \label{asmp:isop}
  Any isoperimetric domain in $X$ is
  either the sub- or super-level set of $f$.
\end{asmp}

\begin{asmp} \label{asmp:ObsVar}
  The observable $\lambda$-variance of $X$ is attained
  by the function $f$.
\end{asmp}

Under these assumptions, we have the conditions of Theorem \ref{thm:isop-rigidity}
for $d\nu(t) = \varphi(t)^n\,dt$.
Precisely, $X$ satisfies $\IC(\varphi(t)^n\,dt)$
and $f_*\mu_X = \varphi(t)^n\,dt$ is an iso-dominant of $X$.
If $a$ and $b$ are both finite, then we have (1) of Theorem \ref{thm:isop-rigidity}.
If only one of $a$ and $b$ is finite, then we have (2).
If $a$ and $b$ are both infinite, then we have (3).
In particular, we observe that $\varphi(t) \to 0$ as $t \to a$ (resp.~$t \to b$)
if $a$ (resp.~$b$) is finite,
because the minimizer (resp.~maximizer) of $f$ is unique if any.

Assumption \ref{asmp:isop}
is satisfied in the following case.
$F = S^1 = \{\,e^{i\theta} \mid \theta \in \R\,\}$, $d\mu_F(\theta) = d\theta/2\pi$, $a = -b < b < +\infty$, $\varphi(-t) = \varphi(t)$ for $t \in [\,0,b\,)$,
and the Gaussian curvature $K(t) = -\varphi''(t)/\varphi(t)$
is positive and strictly monotone decreasing in $(\,a,0\,]$.
Note that $K(-t) = K(t)$ for any $t \in (\,a,b\,)$.
By Ritor\'e's result \cite{Rt:isop-rev} we have Assumption \ref{asmp:isop} in this case.
If in addition the diameter of $X$ is equal to $2b$,
then we also have Assumption \ref{asmp:ObsVar} by Theorem \ref{thm:diam}.

Refer to \cite{Bd:warp, MHH:isop-rev}
for further potential examples of warped products.

\subsection{Non-splitting manifold containing a straight line}
\label{ssec:non-splitting}
Applying Proposition \ref{prop:Cheeger},
we obtain an example of a complete Riemannian manifold $X$
with a fully supported Borel probability measure such that
\begin{enumerate}
\item $X$ satisfies $\IC(\nu)$,
\item $X$ contains a straight line,
\item $X$ is not homeomorphic to $Y \times \R$ for any manifold $Y$.
\end{enumerate}
In fact, there is a complete Riemannian manifold $X$ satisfying (2), (3),
and with positive Cheeger constant (see, for example, \cite{Bs:Ch-const}).
By Proposition \ref{prop:Cheeger}, there is a measure $\nu \in V$
such that $X$ satisfies $\IC(\nu)$.
Note that Corollary \ref{cor:isop-rigidity} proves
\[
\ObsVar_\lambda(X) < \Var_\lambda(\nu).
\]

\section{Appendix: Variance of Spherical Model} \label{sec:Var-sigma}

In this section, we prove \eqref{eq:Var-sigma} and see some
consequent results.  We write $\Var(\cdot) := \Var_{t^2}(\cdot)$
and $\ObsVar(\cdot) := \ObsVar_{t^2}(\cdot)$.

\begin{proof}[Proof of \eqref{eq:Var-sigma}]
For $N\in [\,0,+\infty\,)$, we define
\begin{align*}
  F_N(x) &:= \int_{\frac\pi 2}^x \cos^N t \; dt,\\
  G_N(x) &:= \int_{\frac\pi 2}^x F_N(t)\;dt,\\
  H_N(x) &:= \int_{\frac\pi 2}^x G_N(t)\;dt,\\
  I_N := -F_N(0), &\qquad K_N := -H_N(0).
\end{align*}
We have
\begin{align*}
\int_0^{\frac\pi 2}x^2\cos^N x dx&=\int_0^{\frac\pi 2}x^2(F_N(x))' dx\\
&=[x^2 F_N(x)]_0^{\frac\pi 2} -2\int_0^{\frac\pi 2} x F_N(x) dx\\
&=-2\int_0^{\frac\pi 2} x (G_N(x))' dx\\
&=-2\{[x G_N(x)]_0^{\frac\pi 2}-\int_0^{\frac\pi 2} G_N(x) dx \}\\
&=-2H_N(0)=2K_N.
\end{align*}
For $N \in [\,2,+\infty\,)$, since
\[
F_N(x)=\frac 1N \cos^{N-1}x\sin x+\frac{N-1}N F_{N-2}(x)
\]
we have
\begin{align*}
G_N(x) &= -\frac 1{N^2}\cos^N x+\frac{N-1}N G_{N-2}(x),\\
H_N(x) &= -\frac 1{N^2}F_N(x)+\frac{N-1}N H_{N-2}(x).
\end{align*}
Setting $x=0$, we obtain
\begin{align*}
I_N &= \frac{N-1}N I_{N-2},\\
K_N &= -\frac 1{N^2} I_N +\frac{N-1}N K_{N-2}.
\end{align*}
Therefore, for $N \in (\,0,+\infty\,)$,
\[
I_N=I_{2h} \cdot \prod_{i=0}^{\lceil\frac N2\rceil -1} \frac {N-2i-1}{N-2i},
\]
where $h := N/2-\lceil N/2\rceil +1$.
For $N \in (\,0,+\infty\,)$, we define
\[
S_N := \sum_{i=0}^{\lceil N/2\rceil -1}\frac{1}{(N-2i)^2}.
\]
This satisfies $S_N=S_{N-2}+ 1/N^2$ for $N\in (2,\infty)$.
Since
\begin{align*}
K_N+S_N I_N&=\frac{N-1}N K_{N-2}+(S_N-\frac 1{N^2})I_N\\
&=\frac{N-1}N (K_{N-2}+S_{N-2}I_{N-2}),
\end{align*}
we have, for $N \in (\,0,+\infty\,)$,
\begin{align*}
K_N+S_N I_N&=(K_{2h}+S_{2h} I_{2h})\prod_{i=0}^{\lceil\frac N2\rceil-1}\frac{N-2i-1}{N-2i} \\
&=(K_{2h} I_{2h}^{-1}+S_{2h})I_N,
\end{align*}
so that
\[
\Var(\sigma^{N+1})=2\frac{K_N}{I_N}=2(K_{2h} I_{2h}^{-1} + S_{2h} - S_N).
\]
Putting $k := \lceil\frac N2\rceil-1-i$ in the definition of $S_N$ yields
\[
S_N = \frac 14\sum_{k=0}^{\lceil\frac N2\rceil-1}\frac{1}{(h+k)^2},
\]
which converges to $\frac 14\zeta(2,h)$ as $N \to +\infty$.
We see that
\begin{align*}
\mathrm{Var}(\sigma^{N+1})&=\frac 1{I_N}\int_0^{\frac\pi 2}x^2\cos^N x dx\\
&\leq \frac{\pi^2}{4}\cdot \frac 1{I_N}\int_0^{\frac\pi 2}\sin^2x\cos^Nx dx\\
&=\frac{\pi^2}{4}\cdot \frac{I_N-I_{N+2}}{I_N}\\
&=\frac{\pi^2}{4}\left(1-\frac{I_{N+2}}{I_N}\right)\\
&=\frac{\pi^2}{4}\left(1-\frac{N+1}{N+2}\right)\to 0 \quad \text{as $N\to\infty$}.
\end{align*}
Thus we have $K_{2h} I_{2h}^{-1} + S_{2h} =\frac 14\zeta(2,h)$
and, for $N \in (\,0,+\infty\,)$,
\[
\mathrm{Var}(\sigma^{N+1})=\frac 12 \left(\zeta(2,h) -4S_N\right).
\]
This completes the proof of \eqref{eq:Var-sigma}.
\end{proof}

From \eqref{eq:Var-sigma}, we observe that
\[
\lim_{N\to +\infty} \frac{\Var(\sigma^N)}{\sqrt{N}} = 1,
\]
which is also verified by the Maxwell-Boltzmann distribution law
(see \cite{Sy:mmg}*{Proposition 2.1})
in the case where $N$ runs over only positive integers.

We also have, for any integer $m \ge 2$,
\begin{align*}
\ObsVar(S^m(1)) = \Var(\sigma^m) =
\begin{cases}
 \frac{\pi^2}{12}-\sum_{k=1}^{n-1}\frac 2{(2k)^2} & \text{if $m = 2n-1$},\\
 \frac{\pi^2}{4}-\sum_{k=1}^n\frac 2{(2k-1)^2} &\text{if $m = 2n$}.
\end{cases}
\end{align*}

\end{document}